\documentclass[11pt]{amsart} 
\usepackage{amsmath,amsfonts,amssymb,amsthm,amscd}
\usepackage{color, hyperref, graphicx}
\usepackage{dsfont}
\usepackage[all,color]{xy}

\newtheorem{theorem}{Theorem}[section]

\newtheorem{prop}[theorem]{Proposition}

\newtheorem{fact}[theorem]{Fact}
\newtheorem{cor}[theorem]{Corollary}
\newtheorem{proposition}[theorem]{Proposition}
\newtheorem{lemma}[theorem]{Lemma}

\theoremstyle{definition}

\theoremstyle{remark}
\newtheorem{rem}[theorem]{Remark}
\newtheorem{remark}[theorem]{Remark}

\newcommand{\N}{\mathbb{N}}

\newcommand{\R}{\mathbb{R}}       
\newcommand{\C}{\mathbb{C}}      

\newcommand{\K}{\mathbb{K}} 
              
\newcommand{\ff}{\varphi}                   

\newcommand{\sub}{\subseteq}

\newcommand{\CM}{\mathcal M}

\global\long\def\GL{\mathrm{GL}}
\global\long\def\id{\mathrm{id}}

\DeclareMathOperator{\rank}{rank}

\newcommand{\suchthat}{\;\ifnum\currentgrouptype=16 \middle\fi|\;}

\title{Locally  $\C$-Nash groups}
\author{El\'ias Baro}
\address{Departamento de \'Algebra, Facultad de Matem\'aticas, Universidad Complutense de Madrid, 
28040 Madrid, Spain. 
E-mail address: ebaro@ucm.es}
\author{Juan de Vicente}
\address{Departamento de Matem\'aticas, Universidad Aut\'onoma de Madrid, 28049 Madrid, Spain.
E-mail address: juan.devicente@uam.es}
\author{Margarita Otero}
\address{Departamento de Matem\'aticas, Universidad Aut\'onoma de Madrid, 28049 Madrid, Spain. 
E-mail address: margarita.otero@uam.es}

\thanks{The first and third authors are partially supported by  Spanish MTM2014-55565-P and Grupos UCM 910444. Second author also supported by a grant of the International
Program of Excellence in Mathematics at Universidad Aut\'onoma de Madrid.}


\date{1 August 2018}

\begin{document}
\begin{abstract} We introduce and study the category of locally $\mathbb{C}$-Nash groups, basic examples of such groups are complex algebraic groups. We prove that if two complex algebraic groups are  locally $\C$-Nash isomorphic then they also are biregularly isomorphic. We also show that  both, abelian locally Nash   and abelian locally $\mathbb{C}$-Nash groups, can be characterised via
meromorphic maps admitting an algebraic addition theorem;  we give an invariant of such groups associated to the groups of periods of a chart at the identity. Finally, we prove that the category of simply connected  abelian locally $\mathbb{C}$-Nash groups coincides with that of universal coverings of the abelian complex irreducible algebraic groups  (a complex version of  a  result of Hrushovski and Pillay in \cite{Hrushovski_Pillay, Hrushovski_Pillay_Errata}). 
\end{abstract}
\maketitle
\section{Introduction}
The (real) locally Nash category was first introduced by M.\,Shiota in \cite{Shiota1},  and it was used in the context of groups by J.J.\,Madden and C.M.\,Stanton \cite{Madden_Stanton} (see also  \cite{Shiota2} and \cite{Kawakami96}) to study universal coverings of Nash groups and  to  classify one-dimensional Nash groups. In particular, a category where universal coverings  of real algebraic groups could be defined. Later, E. Hrushovski and A. Pillay showed (\cite{Hrushovski_Pillay, Hrushovski_Pillay_Errata} based on Hrushovski's group configuration theorem \cite{Hrushovski}) that locally Nash groups fulfilled this role perfectly, since they are precisely quotients of universal coverings of 
real algebraic groups by discrete subgroups.  Therefore, a classification of locally Nash groups will provide a first step towards a classification of Nash groups and of real algebraic groups.  

We will consider the analogous category in the complex case.
Such category is based on the notion of $\mathbb{C}$-Nash map, firstly considered by \L ojasiewicz, and then applied  in the solution of approximation 
problems by J.\,Adamus and S.\,Randriambololona \cite{Adamus_Randriambololona}, A.\,Tancredi and A.\,Tognoli \cite{Tancredi_Tognoli} 
and others.
Generalizations of this concept for algebraically closed fields of characteristic $0$ have also been considered by Y.\,Peterzil and 
S.\,Starchenko \cite{Peterzil_Starchenko_1,Peterzil_Starchenko_2} and M.\,Knebusch and R.\,Huber \cite{Huber_Knebusch,Knebusch}. Complex Nash groups were also considered by P.\,Flondor and  F.\,Guaraldo in \cite{FlodorG04}.

Many of the results  we prove are common to both  the Nash and $\C$-Nash  categories, thus,  we have chosen to  speak  about  $\mathbb{K}$-Nash when   talking  about both cases simultaneously.  Paying  attention to both cases will enable  us to apply the results of this paper  to the classification of  two-dimensional abelian locally $\K$-Nash groups in \cite{BDO2ALKNG}.

The main results of the paper are the following. We first introduce the category of locally $\mathbb{C}$-Nash groups, examples of such groups are  complex algebraic groups. We prove that the category of algebraic groups is a full subcategory of the category of locally $\C$-Nash groups (Corollary\,\ref{ALGisomorfismo}). Then, we concentrate on  the abelian case. We show that  abelian locally $\mathbb{K}$-Nash groups can be characterised via global
meromorphic maps admitting an algebraic addition theorem (Theorem\,\ref{T1});  we study an invariant associated to the groups of periods of a chart at the identity (Proposition\,\ref{different ranks}), and finally we prove  a complex version of the mentioned result of Hrushovski and Pillay. Namely, the category of simply connected  abelian locally $\mathbb{C}$-Nash groups equals that of universal coverings of the abelian complex irreducible algebraic groups (Corollary\,\ref{equalcat}). 

As it is mentioned in \cite{Hrushovski_Pillay} for Nash groups, locally Nash $\K$-groups  might fold into the henselian groups category introduced by Perrin (see \cite{Perrin}). If this would be the case, Corollary\,\ref{equalcat} (as well as results  in \cite{Hrushovski_Pillay}) could follow from the work of Perrin. Here, we prove Corollary\,\ref{equalcat} via an extension result from \cite{BDOAAT} that provides a characterisation of simply connected abelian locally  Nash groups (see Theorem\,\ref{T1} and Fact 4.1), and which cannot be deduced from \cite{Perrin}. Having that theorem in hand our proof is elementary,  and it does not make use of neither model theoretic techniques nor schemes, as the work in \cite{Hrushovski_Pillay} and \cite{Perrin}, respectively, do.

The results of this paper are part of the second author's Ph.D. thesis \cite{tesis}.

\smallskip

\emph{All along this paper $\mathbb{K}$ will be $\mathbb{C}$ or $\mathbb{R}$. When we refer to semialgebraic subsets of $\mathbb{C}^n$, we mean semialgebraic as subsets of $\mathbb{R}^{2n}$ 
and similarly for semialgebraic maps. We recall  that -- by definition -- a semialgebraic map has a semialgebraic domain of definition.}

\section{Locally $\K$-Nash manifolds}\label{LKNmap}
Let $U$, $V$ and $W$ be open subsets of $\mathbb{K}^p$ with $W\sub U\cap V$.
Given maps $f:U\rightarrow \mathbb{K}^m$ and $g:V\rightarrow \mathbb{K}^n$ we say that $f$ is \emph{algebraic} over 
$\mathbb{K}(g):=\mathbb{K}(g_1,\ldots ,g_n)$ on $W$ if each of its components $f_1,\ldots ,f_m$ is algebraic over $\mathbb{K}(g)$ on $W$, that is,  
for each $i\in \{1,\ldots ,m\}$ there exists a polynomial 
$P_i\in \mathbb{K}[X_1,\ldots ,X_n,Y]$ of positive degree in $Y$ such that 
\[
P_i(g_1(u),\ldots ,g_n(u), f_i(u))=0 \text{ for each } u\in W. 
\]
In the real context, Nash maps are those that are both smooth and semialgebraic.
Alternatively, Nash maps are those whose domain of definition is an open semialgebraic set and that are both analytic and algebraic 
over $\mathbb{R}(\id)$ (see, \emph{e.g.}, J. Bochnak, M. Coste, M.-F. Roy \cite[Prop.\,8.1.8]{Bochnak_Coste_Roy}).
This definition can be extended to the complex case as follows.

Let $U$ be an open subset of $\mathbb{K}^n$. 
We say that a map $f:U\rightarrow \mathbb{K}^m$ is a \emph{$\mathbb{K}$-Nash map} if $U$ is a semialgebraic, $f$ is analytic (as a $\mathbb{K}$-map)
and algebraic over $\mathbb{K}(\id)$ on $U$.

We remark that although the real and complex case have analogous definitions and share similar properties, 
there are also some subtle differences, because the complex case is more rigid than the real one.
For example, in the real case any semialgebraic function becomes a Nash function, when restricted to an adequate open dense subset of its domain of definition.

\begin{fact}{\cite[2.4.1]{Fernando_Gamboa_Ruiz}}\label{Nash cell decomposition}
Let $U$ be an open subset of $\mathbb{R}^n$ and let $f:U\rightarrow \mathbb{R}^m$ be a semialgebraic map.
Then, there exists an open dense subset $V\sub U$ such that $f:V\rightarrow \mathbb{R}^m$ is a Nash map.
\end{fact}

In contrast with the real case, the following well known result holds only for the complex case.

\begin{fact}[{\cite[Ch.III,\,B.13]{Gunning_Rossi}}]\label{analytic algebraic map}
Let $U$ be an open subset of $\mathbb{C}^n$ and let $f:U\rightarrow \mathbb{C}^m$ be continuous and algebraic over $\mathbb{C}(\id)$.
Then $f$ is analytic.
In particular, if $U$ is a semialgebraic set then $f$ is a $\mathbb{C}$-Nash map.
\end{fact}
We begin by   showing  that the conditions of being algebraic over $\mathbb{K}(\id)$ and being semialgebraic are equivalent 
for analytic functions with a semialgebraic domain of definition.

\begin{proposition}\label{complex Nash}
Let $U$ be an open subset of $\mathbb{K}^n$. 
Then $f:U\rightarrow \mathbb{K}^m$ is both analytic and semialgebraic if and only if $f$ is a $\mathbb{K}$-Nash map.
\end{proposition}
\begin{proof}
For $\mathbb{K}=\mathbb{R}$ this is \cite[Prop.\,8.1.8.]{Bochnak_Coste_Roy}.
The fact that if $f$ is analytic and semialgebraic then $f$ is a $\mathbb{C}$-Nash map is proved in \cite{FLR87}. We prove  the converse. 
We may assume that $m=1$.

It is enough to show that if $f:U\rightarrow \mathbb{C}$ is a continuos function on a 
semialgebraic open subset $U$ of $\mathbb{C}^n$ such that $P(x,f(x))=0$ on $U$ for some $P\in \mathbb{C}[X_1,\ldots ,X_n,X]$, $P\neq 0$, 
then $f$ is a semialgebraic function.
Firstly, we note that \[
A:=\{ (x,y,u,v)\in \mathbb{R}^{n}\times \mathbb{R}^n\times \mathbb{R}\times \mathbb{R} \suchthat P(x+iy,u+iv)=0\} 
\]
is a semialgebraic set.
Also  that the graph of $f$ is a subset of $A$, so to show that $f$ is semialgebraic, it suffices to find a partition of $A$ into semialgebraic
subsets compatible with the graph of $f$.
With that aim,  take $N\in \mathbb{N}$ such that for each $z\in \mathbb{C}^n$ the number of roots of $P(z,X)$ is bounded by $N$.
Fix $k\leq N$ and let 
\[
Z_k:=\{ (x,y)\in \mathbb{R}^{n}\times \mathbb{R}^n :|A\cap ((x,y)\times \mathbb{R}^2)|=k\}.
\]
Since $A$ is a semialgebraic set, $Z_k$ is also a semialgebraic set.
Now, for each $1\leq j\leq k$, let $s_j:Z_k\rightarrow \mathbb{C}$ be the semialgebraic map whose graph is
\begin{multline*}
\{ (z,r)\in (Z_k\times \mathbb{R}^2)\cap A  \suchthat \exists r_1 \ldots \exists r_k\in \mathbb{R}^2, r_1<\ldots <r_k,\\
(z,r_1)\ldots ,(z,r_k)\in (Z_k\times \mathbb{R}^2)\cap A \text{ and } r=r_j \},
\end{multline*}
where $<$ denotes the lexicographic order of $\mathbb{R}^2$.
Take a cell decomposition of $\mathbb{R}^{2n}$ compatible with the sets $Z_1,\ldots ,Z_N, U$ such that each $s_j$ is continuous in each 
one of the cells.
So take $C\sub U$ one of these cells and for each $j\in \{1,\ldots ,k\}$ let $C_j:=\{ x\in C \suchthat s_j(x)=f(x)\}$.
Since, for each $j$, both $f$ and $s_j$ are continuous in $C$, each $C_j$ is closed in $C$.
Also, by definition of the $s_j$, the sets $C_1,\ldots ,C_k$ are disjoint and, as there is a finite number of them, they are open in $C$.
Since $C$ is connected and not empty, we deduce that there exists a unique $j\in \{1,\ldots ,k\}$ such that $C=C_j$.
So $f\equiv s_j$ on $C$, which shows that the restriction of $f$ to $C$ is semialgebraic.
This is clearly enough, since we have partitioned $\mathbb{R}^{2n}$ into finitely many semialgebraic sets $Z_1,\ldots ,Z_N$, and each 
one of these sets into finitely many cells.
\end{proof}

Throughout this paper, we will think of $\mathbb{K}$-Nash maps as analytic,  algebraic  and defined on a semialgebraic set.

The following well known  properties will be frequently used, see, \emph{e.g.}, \cite{Bochnak_Coste_Roy} for the real case, the complex case is proved similarly.

\begin{lemma}\label{basic}
$(1)$ The restriction of a $\mathbb{K}$-Nash map to an open semialgebraic subset of its domain of definition 
 is a $\mathbb{K}$-Nash map.\\
$(2)$  The composition of $\mathbb{K}$-Nash maps is a $\mathbb{K}$-Nash map.\\
$(3)$ Given a $\mathbb{K}$-Nash map $f:U\times V\sub \mathbb{K}^m\times \mathbb{K}^n\rightarrow \mathbb{K}^p$ and 
 $a\in V$, the evaluation  at $a$, $f(x,a):U\rightarrow \mathbb{K}^p$ is also a $\mathbb{K}$-Nash map.\\
$(4)$ If  a $\mathbb{K}$-Nash map $f:U\sub \mathbb{K}^n\rightarrow \mathbb{K}^n$ is an analytic diffeomorphism 
 then $f^{-1}$ is also a $\mathbb{K}$-Nash map.
\end{lemma}

Let $U$ be an open subset of $\mathbb{K}^m$.
We say that $f:U\rightarrow V\sub \mathbb{K}^n$ is a $\mathbb{K}$-\emph{Nash diffeomorphism} if $f$ is  both an analytic 
diffeomorphism and a $\mathbb{K}$-Nash map.
Let $M$ be a $\mathbb{K}$-analytic manifold. 
Two charts $(U,\phi )$ and $(V,\psi )$ of an atlas for $M$ are $\mathbb{K}$-\emph{Nash compatible} if $\phi (U)$ and $\psi (V)$ are semialgebraic and either $U\cap V= \emptyset$ or 
\[
\psi \phi ^{-1}:\phi (U\cap V)\rightarrow \psi (U\cap V) 
\]
is a $\mathbb{K}$-Nash diffeomorphism.
An atlas of $M$ is a $\mathbb{K}$-\emph{Nash atlas} if any two charts in the atlas are $\mathbb{K}$-Nash compatible.
In particular, $\phi (U)$ is semialgebraic for any $(U,\phi )$ in the $\mathbb{K}$-Nash atlas.
A $\mathbb{K}$-analytic manifold $M$ together with a $\mathbb{K}$-Nash atlas is called a \emph{locally $\mathbb{K}$-Nash 
manifold} (the word ``locally'' is added since the term $\mathbb{K}$-Nash manifold is usually reserved for those locally 
$\mathbb{K}$-Nash manifolds that admit a finite $\mathbb{K}$-Nash atlas).

Let $M_1$ and $M_2$ be locally $\mathbb{K}$-Nash manifolds equipped with $\mathbb{K}$-Nash atlases $\{(U_i,\phi _i)\}_{i\in I}$ and $\{(V_j,\psi _j)\}_{j\in J}$, respectively. A \emph{locally $\mathbb{K}$-Nash map} $f:M_1\rightarrow M_2$ is a (continuous) map such that for every $a\in M_1$ and every $j\in J$ such that $f(a)\in V_j$ there exists $i\in I$ and an open subset $U\sub U_i$ such that $a\in U$, $f(U)\sub V_j$ and
\[
\psi _j f\phi _i ^{-1} : \phi _i  (U) \rightarrow \psi _j (V_j)
\]
is a $\mathbb{K}$-Nash map. A locally $\mathbb{K}$-Nash map $f:M_1\rightarrow M_2$ is a \emph{locally $\mathbb{K}$-Nash diffeomorphism} if $f$ is  both an 
analytic (global) diffeomorphism and  a locally $\mathbb{K}$-Nash map (or equivalently  $f$ and $f^{-1}$  locally $\mathbb{K}$-Nash maps).

Locally $\mathbb{K}$-Nash maps can be characterised as follows.

\begin{proposition}\label{characterisation of locally Nash maps} Let $M_1$ and $M_2$ be locally $\mathbb{K}$-Nash manifolds with $\mathbb{K}$-Nash atlases $\{(U_i,\phi _i)\}_{i\in I}$ and $\{(V_j,\psi _j)\}_{j\in J}$ respectively.
The following are equivalent:\\
$(1)$ $f:M_1\rightarrow M_2$ is a locally $\mathbb{K}$-Nash map.
\\ $(2)$ For every $a\in M_1$ and for each $i\in I$ and $j\in J$ such that $a\in U_i$ and $f(a)\in V_j$ 
 there exists an open subset $U$ of $U_i$ such that $a\in U$, $f(U)\sub V_j$, and 
 \[ 
 \psi _j f  \phi _i^{-1} :\phi _i(U)\rightarrow \psi _j(V_j) 
 \]
 is a $\mathbb{K}$-Nash map.
\\ $(3)$ For every $a\in M_1$ there exist $i \in I$ and $j\in J$ such that $a\in U_i$ and $f(a)\in V_j$ and 
 there exists an open subset $U$ of $U_i$ such that $a\in U$, $f(U)\sub V_j$, and 
 \[
 \psi_j   f   \phi_i^{-1}:\phi_i(U )\rightarrow \psi_j (V_j ) 
 \] 
 is a $\mathbb{K}$-Nash map.
\end{proposition}
\begin{proof}
It suffices to prove that $(3)$ implies $(2)$.
Fix $a\in M_1$ and let $i \in I$, $j\in J$ and $U\sub U_i$  as in  $(3)$.
Fix $k\in I$ and $\ell \in J$ with $a\in U_k$ and $f(a)\in V_\ell$.
Clearly, it suffices to show that there exists an open subset $U'$ of $U_k$ with $a\in U'$ such that 
\[
\psi_\ell   f   \phi_k^{-1}:\phi_k(U')\rightarrow \psi_\ell (V_\ell) 
\]
is $\mathbb{K}$-Nash. To prove the latter, firstly note that  by continuity  exists an open subset $U'$ of $U\cap U_k\sub U_i\cap U_k$ with $a\in U'$ such that
\[
(\psi_j   f   \phi_i^{-1})(\phi_i(U'))\sub \psi_j(V_j\cap V_\ell). 
\]
Moreover, we can assume that $\phi_i(U')$ is semialgebraic (it suffices to take, instead of $U'$, the preimage of an open ball 
centered in $\phi_i(a)$ and contained in the original $\phi_i(U')$), thus  the map 
\[
\psi_j   f   \phi_i^{-1}:\phi_i(U')\rightarrow \psi_j (V_j\cap V_\ell) 
\]
is  a $\mathbb{K}$-Nash map. On the other hand,  since change of charts 
are $\mathbb{K}$-Nash maps, the composition 
\[
\psi_\ell   f   \phi_k^{-1}=(\psi_\ell   \psi_j^{-1})\circ (\psi_j f   \phi_i^{-1})\circ (\phi_i   \phi_k^{-1})
: \phi _k(U')\rightarrow \psi _\ell(V_\ell)
\]
is also a $\mathbb{K}$-Nash map, as required.
\end{proof}

\section{Locally \texorpdfstring{$\mathbb{K}$}{K}-Nash groups}\label{LKNgroup} 

A \emph{locally $\mathbb{K}$-Nash group} is a locally Nash manifold equipped with group operations -- multiplication and inversion -- which 
are  locally $\mathbb{K}$-Nash maps. A \emph{homomorphism} (\emph{isomorphism})  of locally $\mathbb{K}$-Nash groups is a locally $\mathbb{K}$-Nash map (resp. diffeomorphism) that is also a homomorphism (resp. isomorphism) of groups. 

We begin by  showing  how to describe the locally $\mathbb{K}$-Nash structure of a locally $\mathbb{K}$-Nash group via a chart of the identity. Let $G$ together with an analytic atlas $\mathcal{A}$ be a $\mathbb{K}$-analytic group --  a Lie group -- 
and let $(U,\phi)$ be a chart  of the identity in $\mathcal{A}$. Then, 
\[
\mathcal{A}_{(U,\phi)}:= \{ (gU, \phi _g) \suchthat \phi _g:gU\rightarrow \mathbb{K}^n:u\mapsto \phi (g^{-1}u)\}_{g\in G}
\]
is also an analytic atlas for $G$. $\mathcal{A}_{(U,\phi)}$ might not be a $\mathbb{K}$-Nash atlas for $G$, but if it is so, then \emph{the locally $\mathbb{K}$-Nash group $G$ equipped with $\mathcal{A}_{(U,\phi)}$ will be denoted $(G,\cdot,\phi |_U)$}.

\begin{fact}{\cite[Lem.\,1]{Madden_Stanton}}\label{compatibility0}
Let $G$ be a $\mathbb{K}$-analytic group with atlas $\mathcal{A}$.
Let $(U,\phi )\in \mathcal{A}$ be a chart of the identity such that:
\begin{enumerate}
\noindent \begin{minipage}{0.88\textwidth}
\item[$(i)$] there exists an open neighbourhood of the identity $U'\sub U$ such that
\[
\phi\,m(\phi ^{-1},\phi ^{-1}):\phi (U')\times \phi (U')\rightarrow \phi (U) :
(x,y)\mapsto \phi (\phi ^{-1}(x)\cdot \phi ^{-1}(y))
\]
is a $\mathbb{K}$-Nash map, where $m:G\times G\to m: (g,h)\mapsto g\cdot h$, and
\item[$(ii)$] for each $g\in G$ there exists an open neighbourhood of the identity $U_g\sub U$ such that
\[
\phi \circ -^{g} \circ \phi ^{-1}:\phi (U_g)\rightarrow \phi (U) :
x\mapsto \phi (g^{-1}\phi ^{-1}(x)g)
\]
is a $\mathbb{K}$-Nash map.
\end{minipage}
\end{enumerate}
Then, there exists a neighbourhood of the identity $V\sub U$ such that 
$\mathcal{A}_{(V,\phi )}=\{ (gV,\phi _g)\} _{g\in G}$ is a $\mathbb{K}$-Nash atlas for $G$ and hence $(G,\cdot ,\phi |_V)$ is a 
locally $\mathbb{K}$-Nash group.
\end{fact}

We note that if $G$ is an abelian group then $(ii)$ of  the fact is trivially satisfied.
So, in that case, the proposition says that each chart of the identity satisfying $(i)$ induces a locally $\mathbb{K}$-Nash group 
structure on $G$.

\begin{proposition}\label{compatibility1}
Let $G$ be a locally $\mathbb{K}$-Nash group equipped with a $\mathbb{K}$-Nash atlas $\mathcal{A}$.
Then, for every chart of the identity $(U,\phi )\in \mathcal{A}$, there exists an open neighbourhood of the identity $V\subseteq U$ such that $G$ 
equipped with $\mathcal{A}$ is isomorphic to $(G,\cdot , \phi |_V)$.
\end{proposition}
\begin{proof}
Firstly, we will check that $(U,\phi)$ satisfies $(i)$ and $(ii)$ of Fact\,\ref{compatibility0}.
So that, we will be able to conclude  that there exists $V\sub U$ such that $\mathcal{A}_{(V,\phi )}$ is a $\mathbb{K}$-Nash atlas for $G$. Finally, we will show that the identity map from $G$ equipped with $\mathcal{A}$ to $G$ equipped with 
$\mathcal{A}_{(V,\phi )}$ is a locally $\mathbb{K}$-Nash diffeomorphism and, hence, an isomorphism of locally $\mathbb{K}$-Nash groups.

Let $(U,\phi )\in \mathcal{A}$ be a chart of the identity.
Since $m:G\times G\rightarrow G$ is a locally $\mathbb{K}$-Nash map when $G$ is equipped with $\mathcal{A}$, by Proposition\,\ref{characterisation of locally Nash maps}.$(2)$ we deduce the following facts:

\noindent 
(1) There exists an open neighbourhood of the identity $U'\sub U$ such that
\[
\phi\,m (\phi ^{-1},\phi ^{-1}):\phi (U')\times \phi (U')\rightarrow \phi (U) :
(x,y)\mapsto \phi (\phi ^{-1}(x)\cdot \phi ^{-1}(y))
\]
is a $\mathbb{K}$-Nash map.
So $(U,\phi )$ satisfies $(i)$ of Fact\,\ref{compatibility0}.

\noindent
(2) Fix $g\in G$ and $(W_1,\psi _1),(W_2,\psi _2)\in \mathcal{A}$ coordinate neighbourhoods of $g$ and $g^{-1}$ respectively.
Then there exist open neighbourhoods $W_1'\sub W_1$ and $W_2'\sub W_2$ of $g$ and $g^{-1}$ respectively such that
$$\displaystyle 
\begin{array}{rcrcl}
\phi\,m (\psi _2 ^{-1},\psi _1 ^{-1})& : & \psi _2 (W_2')\times \psi _1 (W_1')& \rightarrow & \phi (U) \\
& & (z,x) & \mapsto & \phi (\psi _2 ^{-1}(z)\cdot \psi _1 ^{-1}(x))
\end{array}$$
is a $\mathbb{K}$-Nash map.
Similarly, there exist open neighbourhoods $U_g\sub U$ and $W_1''\sub W_1'$ of the identity and $g$ respectively such that 
$$\displaystyle 
\begin{array}{rcrcl}
\psi _1 m (\phi ^{-1},\psi _1 ^{-1})& : & \phi (U_g)\times \psi _1 (W_1'')& \rightarrow & \psi _1 (W_1') \\
& & (x,y) & \mapsto & \psi _1 (\phi ^{-1}(x)\cdot \psi _1 ^{-1}(y))
\end{array}$$
is a $\mathbb{K}$-Nash map.
If we evaluate the first map at $z=\psi _2 (g^{-1})$ and the second at $y=\psi _1 (g)$, we again obtain $\mathbb{K}$-Nash 
maps.
Then, composing both maps, we deduce that $(U,\phi )$ satisfies $(ii)$ for $g$ of Fact\,\ref{compatibility0}.

Hence, $(U,\phi)$ is under the hypothesis of Fact\,\ref{compatibility0} and, therefore, there exists an open neighbourhood of the 
identity $V\sub U$ such that $\mathcal{A}_{(V,\phi )}$ is a $\mathbb{K}$-Nash atlas for $G$.

Now, we check that the identity map from $G$ equipped with $\mathcal{A}$ to $G$ equipped with $\mathcal{A}_{(V,\phi )}$ is a 
locally $\mathbb{K}$-Nash diffeomorphism. It is enough to show that it is a locally $\mathbb{K}$-Nash map.
By definition it suffices to prove that for each $g,h\in G$ with $g\in hV$ there exist $(W_1,\psi _1 )\in \mathcal{A}$ with 
$g\in W_1$ and an open neighbourhood $W_1'\sub W_1\cap hV$ of $g$ such that ($\psi _1(W_1')$ is semialgebraic and)
\[
\phi _h   \psi _1 ^{-1} : \psi _1(W_1') \rightarrow \phi (V): x\mapsto \phi (h^{-1} \psi _1 ^{-1}(x))
\]
is a $\mathbb{K}$-Nash map.
Let $g$ and $h$ be fixed with $g\in hV$. Let $(W_2,\psi _2 )\in \mathcal{A}$ be a coordinate neighbourhood of $h^{-1}$.
Assume $G$ is equipped with $\mathcal{A}$. Since $h^{-1}g\in V$ and $m :G\times G\rightarrow G$ is a locally $\mathbb{K}$-Nash map, there exists a coordinate neighbourhood $(W_1,\psi _1)\in \mathcal{A}$ of $g$, and open neighbourhoods $W_2'\sub W_2$ and $W_1'\sub W_1$ of $h^{-1}$ and $g$ respectively such that $W_2'\cdot W_1'\sub V$ and
\[
\begin{array}{rcrcl}	
\phi\,m (\psi _2 ^{-1},\psi _1 ^{-1})& : & (\psi _2 (W_2'), \psi _1(W_1')) & \rightarrow & \phi (V)\\  
& & (x,y) & \mapsto & \phi (\psi _2 ^{-1}(x)\cdot \psi _1 ^{-1}(y))
\end{array}
\]
is a $\mathbb{K}$-Nash map. We evaluate the map above at $x=\psi _2 (h^{-1})$ to deduce that
\[
\phi _h   \psi _1 ^{-1}: \psi _1 (W_1')\rightarrow \phi (V):x \mapsto \phi (h^{-1} \psi _1 ^{-1}(x))
\]
is a $\mathbb{K}$-Nash map, as required.
\end{proof}

Next proposition provides a sufficient condition for a pure group homomorphism of locally $\mathbb{K}$-Nash groups to be a 
locally $\mathbb{K}$-Nash homomorphism.

\begin{proposition}\label{lochomo} Let $G$ and $H$ be locally $\mathbb{K}$-Nash groups and let $\alpha :G\rightarrow H$ be a group homomorphism.
Suppose there are charts $(U,\phi)$ and $(V,\psi)$ of $G$ and $H$ respectively and an open subset $U'\sub U$ such that 
$\alpha (U')\sub V$ and $$\psi   \alpha    \phi^{-1}:\phi(U')\rightarrow \psi(V)$$ is a $\mathbb{K}$-Nash map.
Then $\alpha $ is a locally $\mathbb{K}$-Nash homomorphism. 
\end{proposition}
\begin{proof}
Firstly, we prove that $\alpha $ is a locally $\mathbb{K}$-Nash map, provided that $U'\sub U$ and $V$ are neighbourhoods of the identity of 
$G$ and $H$ respectively.
By Proposition\,\ref{compatibility1}, we can assume that the locally $\mathbb{K}$-Nash groups $G$ and $H$ equipped with 
$\mathcal{A}_{(U ,\phi)}$ and $\mathcal{A}_{(V ,\psi)}$, respectively, are locally $\mathbb{K}$-Nash isomorphic to the original structures.
Let $g\in G$.
We have that $(gU,\phi_g)$ and $(\alpha (g)V,\psi_{\alpha (g)})$ are charts of $G$ and $H$ respectively with $g\in gU'\sub gU$ and 
$\alpha (g)\in \alpha (g)V$.
By Proposition\,\ref{characterisation of locally Nash maps}.$(3)$, it would be enough to show that the map
\[
\psi_{\alpha (g)}  \alpha   \phi^{-1}_g:\phi(U') \rightarrow \psi(V) 
\]
is $\mathbb{K}$-Nash.
The latter is true since $\psi   \alpha    \phi^{-1}$ is a $\mathbb{K}$-Nash map and
\begin{align*}
(\psi_{\alpha (g)}  \alpha   \phi^{-1}_g)(x) & =(\psi_{\alpha (g)}  \alpha )(g\phi^{-1}(x))\\
& =\psi_{\alpha (g)}(\alpha (g)\alpha (\phi^{-1}(x)))\\
& =\psi(\alpha (g)^{-1}\alpha (g)\alpha (\phi^{-1}(x)))\\
& =\psi(\alpha (\phi^{-1}(x))).
\end{align*}

It remains to prove that we can assume that the relevant open sets can be taken neighbourhoods of the identity.
Fix $g\in U'$.
Since multiplication in  $G$ is a locally $\mathbb{K}$-Nash map, there exist a chart $(U_0,\phi_0)$ of $G$ 
and an open neighbourhood of the identity $U'_0\sub U_0$ such that $gU'_0\sub U'$ and the map
\[
L_g:\phi_0(U'_0)\rightarrow \phi(U'):x\mapsto \phi(g\phi_0^{-1}(x))
\]
is a $\mathbb{K}$-Nash map.
Similarly, there exist a chart $(V_0,\psi_0)$ of the identity of $H$ and an open subset $V'\ni \alpha (g)$ of $V$ such that 
$\alpha (g)^{-1}V'\sub V_0$ and 
\[
L_{\alpha (g)^{-1}}:\psi(V')\rightarrow \psi_0(V_0):x\mapsto \psi_0(\alpha (g)^{-1}\psi^{-1}(x))
\]
is a $\mathbb{K}$-Nash map.
By continuity and since $(\psi   \alpha    \phi^{-1}  L_g)(\phi_0(e))=\psi(\alpha (g))$, we can take $U'_0$ small enough so that 
\[
(\psi   \alpha    \phi^{-1}  L_g)(\phi _0(U'_0))\sub\psi(V').
\]
In particular, the composition
\[
\begin{array}{rcccl}
L_{\alpha (g)^{-1}}  \psi   \alpha    \phi^{-1}  L_g & : & \phi_0(U'_0) & \rightarrow & \psi_0(V_0) \\
& & x & \mapsto & \psi_0(\alpha (\phi^{-1}_0(x)))
\end{array}
\]
is a $\mathbb{K}$-Nash map, as required.
\end{proof}

\begin{remark}\label{improved lochomo}
If $\mathbb{K}=\mathbb{R}$ in Proposition\,\ref{lochomo} then the result holds in case that 
$\psi   \alpha    \phi^{-1}:\phi(U')\rightarrow \psi(V)$ is just a semialgebraic map.
Indeed, by Fact\,\ref{Nash cell decomposition} and restricting $U'$ if necessary, we can assume that 
the map $\psi   \alpha    \phi^{-1}:\phi(U')\rightarrow \psi(V)$ is Nash.
\end{remark}

Next, we characterise those analytic isomorphisms which are isomorphisms of locally $\mathbb{K}$-Nash groups.
We recall that if an analytic map is an isomorphism of groups then its inverse is also an analytic map.

\begin{proposition}\label{compatibilityAut}
Let $G$ and $H$ be locally $\K$-Nash groups equipped with atlases $\mathcal{A}$ and $\mathcal{B}$, respectively. Let  $\alpha :G \rightarrow H$ be a continuous isomorphism.
Then,  the following  are equivalent. 

$(1)$ $\alpha$ is an isomorphism of locally $\mathbb{K}$-Nash groups, and  

$(2)$ there exist,  
for each pair of charts of the identity $(U,\phi)\in\mathcal{A}$ and $(V, \psi) \in \mathcal{B}$, an open neighbourhood of the identity  
$W\sub  U \cap \alpha^{-1}(V)$ such that $\psi  \alpha$ is algebraic over $\mathbb{K}(\phi)$ on $W$. 
\end{proposition}
\begin{proof}
We begin with  (1) implies (2).
Fix a pair of charts of the identity $(U,\phi )\in \mathcal{A}$ and $(V,\psi )\in \mathcal{B}$.
Since $\alpha$ is a locally $\mathbb{K}$-Nash map, by Proposition\,\ref{characterisation of locally Nash maps}.$(2)$ 
there exists an open neighbourhood of the identity $W\sub U\cap \alpha ^{-1}(V)$ such that
\[
\psi   \alpha   \phi ^{-1} : \phi (W) \rightarrow \psi (V): x\mapsto \psi (\alpha ( \phi ^{-1}(x)))
\]
is a $\mathbb{K}$-Nash map.
So $\psi   \alpha   \phi ^{-1}$ is algebraic over $\mathbb{K}(\id)$ on $\phi(W)$ and, hence, $\psi   \alpha$ 
is algebraic over $\mathbb{K}(\phi)$ on $W$, as required.

Now we show  (2) implies (1) for the case $\mathbb{K}=\mathbb{R}$.
Fix $i\in\{1,\ldots,n\}$.
By hypothesis $\psi_i  \alpha$ is algebraic over $\mathbb{R}(\phi)$ on $W$ and, therefore, since $\phi$ is a diffeomorphism,
$\psi_i   \alpha   \phi^{-1}$ is algebraic over $\mathbb{R}(\id)$ on $\phi(W)$. 
Hence, there exists a polynomial $P\in \mathbb{R}[x][Y]$ such that $$P(x,(\psi_i   \alpha   \phi ^{-1})(x)) = 0$$for all 
$x \in \phi(W)$.
Without loss of generality, we can assume that $\phi(W)$ is semialgebraic.
Then, by the proof of \cite[Prop.\,8.1.8]{Bochnak_Coste_Roy} and since 
$\psi_i   \alpha   \phi ^{-1}$ is continuous, we obtain that each coordinate function $\psi_i   \alpha   \phi^{-1}$ is a semialgebraic function on $\phi(W)$.
By Remark\,\ref{improved lochomo}, we deduce that $\alpha$ is a locally Nash map.
Moreover, we also have that the inverse of the above map,
\[
\phi   \alpha^{-1}  \psi^{-1}:\psi(\alpha(W))\rightarrow \phi(W)\sub \phi(U), 
\]
is semialgebraic and, therefore, again by Remark\,\ref{improved lochomo}, we deduce that $\alpha^{-1}$ is a locally Nash map.
Thus, $\alpha$ is a locally Nash isomorphism.  

Finally, we prove (2) implies (1) for the case $\mathbb{K}=\mathbb{C}$.
We can assume that $\phi(W)$ is connected and semialgebraic. 
Since $\psi   \alpha   \phi^{-1}|_{\phi(W)}$ is both continuous and algebraic over $\mathbb{C}(\id)$, 
it follows from Fact\,\ref{analytic algebraic map} and Proposition\,\ref{lochomo} that $\alpha$ is a locally $\mathbb{C}$-Nash 
homomorphism.
Moreover, by invariance of domain, we have that $V':=\alpha (W)$ is an open subset of $V$ and 
$\psi   \alpha   \phi^{-1}|_{\phi(W)}$ an homeomorphism. 
Thus, we have that
\[
\phi   \alpha^{-1}  \psi^{-1}|_{\psi(V')} 
\]
is continuous and also algebraic over $\mathbb{C}(\id)$.
Then, again by Fact\,\ref{analytic algebraic map} and Proposition\,\ref{lochomo}, we conclude that $\alpha^{-1}$ is a locally $\mathbb{C}$-Nash isomorphism, as required. 
\end{proof}

The following is an immediate consequence of  the last proposition.

\begin{cor}\label{compatibilityGerm}
Let $G$ together with a $\mathbb{K}$-Nash atlas $\mathcal{A}$ be a locally $\mathbb{K}$-Nash group.
Let $(U,\phi )$ and $(V,\psi)$ be charts of the identity of $\mathcal{A}$.
If $(G,\cdot,\phi |_U)$ and $(G,\cdot,\psi |_V)$ are locally $\mathbb{K}$-Nash groups then the identity map is a locally $\mathbb{K}$-Nash isomorphism.
\end{cor}

\proof
 Note that since $(U,\phi )$ and $(V,\psi)$ are $\mathbb{K}$-Nash compatible, the result follows from the proposition taking $H=G$ and $\alpha=Id_G$.
 \endproof

In particular,  the corollary  says that  the isomorphism type of the locally $\K$-Nash group $(G,\cdot,\phi |_U)$ depends only on the germ of $\phi $ at the identity.

\

Next we show that the category of locally $\mathbb{K}$-Nash groups is closed under universal coverings 
\begin{proposition}\label{covering}
Let $G$ be a locally $\mathbb{K}$-Nash group.
Let $\widetilde{G}$ be its (analytic) universal cover.
Then, $\widetilde{G}$ can be equipped with the structure of a locally $\mathbb{K}$-Nash group in such a way that  the covering map 
$\pi :\widetilde{G}\rightarrow G$ is a locally $\mathbb{K}$-Nash homomorphism.
This structure is unique up to isomorphism.
Moreover, any analytic section $s:W\to\widetilde{G}$, where $W\subseteq G$ is open, is a locally  $\mathbb{K}$-Nash map.
\end{proposition}
\begin{proof}
By Proposition\,\ref{compatibility1}, we may assume that the $\mathbb{K}$-Nash atlas of $G$ is of the form 
$\mathcal{A}_{(U,\phi )}$ for some open neighbourhood of the identity $U\sub G$, where
\[
\mathcal{A}_{(U,\phi )}:=\{(gU, \phi _g) \suchthat \phi _g: gU\rightarrow \mathbb{K}^n:u\mapsto \phi (g^{-1}u)\}_{g\in G}\ .
\]
We know that $\widetilde{G}$ is an analytic group and $\pi$ an analytic homomorphism (see \cite[Cor.\,2.6.2]{Varadarajan}).
Then,  there exist an open neighbourhood of the identity $U'\sub U$ and an analytic section 
$s :U'\rightarrow \widetilde{G}$ such that $s (U')$ is a neighbourhood of the identity of $\widetilde{G}$ and 
$\pi   s = Id$.
Thus, $(s(U'),\phi \pi)$ is a chart of $\widetilde{G}$.
We note that since $(U,\phi )$ satisfies properties $(i)$ and $(ii)$ of Fact\,\ref{compatibility0}, 
$(s (U'),\phi   \pi)$ also does.
Therefore,  there exists an open neighbourhood of the identity $V\sub s (U')$ such that 
$\widetilde{G}$ together with $\mathcal{A}_{(V,\phi   \pi)}$ is a locally $\mathbb{K}$-Nash group, where
\[
\mathcal{A}_{(V,\phi   \pi )}:=\{(gV, \phi _{\pi(g)}  \pi \}_{g\in \widetilde{G}}\ .
\]
To check that the covering map $\pi$ is a locally $\mathbb{K}$-Nash homomorphism, it is enough to apply Proposition\,\ref{lochomo} to 
$\pi$ noting that
\[
\phi   \pi   (\phi   \pi )^{-1} : \phi   \pi (V) \rightarrow \phi (U)
\]
is the identity map and, hence, a $\mathbb{K}$-Nash map.

Now, we check that the structure is unique up to locally $\mathbb{K}$-Nash isomorphism. 
Fix two locally $\mathbb{K}$-Nash atlas $\mathcal{A}_1$ and $\mathcal{A}_2$ such that 
$\pi :\widetilde{G}\rightarrow G$ is a locally $\mathbb{K}$-Nash homomorphism when $\widetilde{G}$ is equipped with either 
$\mathcal{A}_1$ and $\mathcal{A}_2$.
Since the analytic structure is unique up to analytic isomorphism, it is enough to check that $(\widetilde{G},\mathcal{A}_1)$ and
$(\widetilde{G},\mathcal{A}_2)$ are locally $\mathbb{K}$-Nash isomorphic.
Since $\pi$ is a locally $\mathbb{K}$-Nash map, given a chart of the identity $(\varphi ,W)$ of $G$ there exist charts of the identity 
$(\psi_1 ,V_1)\in \mathcal{A}_1$ and $(\psi_2 ,V_2)\in \mathcal{A}_2$ such that both 
$\varphi   \pi   \psi_1 ^{-1} : \psi_1 (V_1)\rightarrow \mathbb{K}^n$ and 
$\varphi   \pi   \psi_2 ^{-1} :\psi_2 (V_2)\rightarrow \mathbb{K}^n$ are $\mathbb{K}$-Nash diffeomorphisms.
So $\psi_2   \psi_1 ^{-1}:\psi_1 (V_1\cap V_2)\rightarrow \psi_2 (V_1\cap V_2)$ is a $\mathbb{K}$-Nash diffeomorphism.
This means that $\psi_2$ is algebraic over $\mathbb{K}(\psi_1)$ on some open neighbourhood of the identity and, hence, by 
Proposition\,\ref{compatibilityAut}, the identity map from $(\widetilde{G},\mathcal{A}_1)$ to $(\widetilde{G},\mathcal{A}_2)$
is a locally $\mathbb{K}$-Nash isomorphism.

Next, we fix an open subset $W\sub G$ and an analytic section $s':W\rightarrow \widetilde{G}$ and we want to check that 
$s'$ is a locally $\mathbb{K}$-Nash map.
(We note that since $W$ is an open subset of $G$, it has itself the structure of a locally $\mathbb{K}$-Nash manifold.)
By Proposition\,\ref{characterisation of locally Nash maps}.$(3)$, it is enough to check that for each $a\in W$ there exist 
$g\in G$, $\tilde g\in \widetilde{G}$ and $U'\sub U$ such that the restriction of 
$(\phi _{\pi (\tilde g)}  \pi)   s'   \phi _g^{-1}$ to $\phi (U')$ is a $\mathbb{K}$-Nash map.
Taking $g:=a$, $\tilde g:=s'(a)$ and $U'$ such that $a\phi(U')$ is a semialgebraic open subset of $W$, we get that the previous 
map is the identity map on $\phi(U')$ and, hence, a $\mathbb{K}$-Nash map.
\end{proof}

To consider  non simply connected locally $\mathbb{K}$-Nash groups, we will be interested in taking  quotients of locally $\mathbb{K}$-Nash groups by normal discrete subgroups, next remark shows that doing so we keep within  our category.

\begin{remark}\label{quotients by normal subgroups}
Let $(G,\cdot, \phi |_U)$ be a locally $\mathbb{K}$-Nash group and let $\Lambda$ a normal discrete subgroup of $G$.
Then, $(G/\Lambda, \cdot  ,\phi |_U)$ is a locally $\mathbb{K}$-Nash group and the projection map is a locally 
$\mathbb{K}$-Nash map.
\end{remark}
\begin{proof}
We can provide to $G/\Lambda$ a structure of locally $\mathbb{K}$-Nash group as follows.
Let $\pi:G\rightarrow G/\Lambda$ be the canonical projection of $G$ onto $G/\Lambda$.
Take $(U,\phi)$ a chart of $G$ such that $U$ only intersects $\Lambda$ in the identity of $G$ and define
$\psi:\pi(U) \rightarrow \phi(U): \pi (u) \mapsto \phi(u)$.
Since $G/\Lambda$ has a natural structure of (complex or real)  analytic group 
and $(\pi(U),\psi )$ obviously satisfies $(i)$ and $(ii)$ of Fact\,\ref{compatibility0} (because $(U,\phi)$ -- being a chart of $G$ -- satisfies both properties),  there  must exist $V\sub \pi (U)$ such that $\mathcal{A}_{(V,\psi)}$ is a $\mathbb{K}$-Nash 
atlas for $G/\Lambda$.
When $G/\Lambda$ is equipped with this locally $\mathbb{K}$-Nash structure, the projection $\pi$ is a locally $\mathbb{K}$-Nash map.
\end{proof}

\section{Locally $\mathbb{C}$-Nash groups and algebraic groups}
This section is dedicated to  check that there is a natural adaptation to our context of classical results concerning the relation between the algebraic and the analytic structure of complex algebraic groups.

Following the notation of \cite{Poizat}, recall that an (abstract) algebraic group is a group $G$ together with a finite covering of 
subsets $Y_1,\ldots,Y_d$ of $G$ with bijections $f_i$ from $Y_i$ onto an affine algebraic set $X_i$, for $i=1,\ldots ,d$,
such that:

\noindent
(1) For each pair $(i,j)$, the set $X_{i,j}:= f_i(Y_i \cap Y_j)$ is a Zariski open subset of $X_i$ and $f_{i,j} := f_j   f_i^{-1}$ is locally rational.
Recall that locally rational means that for each point of $X_{i,j}$ there exists a smaller Zariski open neighbourhood of the point where the map is regular,
that is, $f_{i,j}$ can be written as the quotient of two polynomials whose denominator does not vanish in that neighbourhood.

\noindent
(2)  For each triple $(i,j,k)$, both $$\{(f_i(y_1),f_j(y_2)) \suchthat y_1y_2\in Y_k\}\text{ and }\{f_i(y) \suchthat y^{-1}\in Y_j \}$$ are respectively
Zariski open subsets of $X_i\times X_j$ and $X_i$ and the maps  
\[
\{(f_i(y_1),f_j(y_2)) \suchthat y_1y_2\in Y_k\} \rightarrow X_k: (f_i(y_1),f_j(y_2)) \mapsto f_k(y_1y_2)
\]
and
\[
\{f_i(y) \suchthat y^{-1}\in Y_j\} \rightarrow X_j: f_i(y) \mapsto f_j(y^{-1})
\]
are locally rational.

We call each pair $(Y_i,f_i)$ a \emph{Zariski chart} of $G$.
A subset $A$ of $G$ is Zariski closed if $f_i(A\cap Y_i)$ is a Zariski closed subset of $X_i$ for each $i=1,...,d$. Given an arbitrary subset $A$ of $G$, there is an smallest closed subset $\overline{A}^\text{zar}$ of $G$ containing $A$, and called the \emph{Zariski closure} of $A$ in $G$ (otherwise, there we would be an infinite descending chain of Zariski closed subsets of  $X_i$ for some $i=1,\ldots,d$, a contradiction). 

We say that an (abstract) algebraic group is irreducible if its corresponding underlying variety is irreducible, that is, if it cannot be written as the union of two proper closed subsets. Finally, we say that \em $G$ is defined over $\R$ \em if all the relevant algebraic sets and locally rational functions are defined with parameters from $\R$. 

We first show that a complex algebraic group has a natural structure of locally $\mathbb{C}$-Nash group.
We recall that, given an affine algebraic set $X\sub \mathbb{C}^m$, the set of smooth points of $X$ is
$$
\text{Smooth}(X)=\{a\in X \suchthat \dim(T_{X,a})=\dim(X):=\min\{\dim(T_{X,b}) \suchthat b\in X\} \},
$$
where $T_{X,a}$ denotes the Zariski tangent space. The set $\text{Smooth}(X)$ is an open Zariski subset of $X$ and it has a natural structure of complex analytic submanifold of $\mathbb{C}^m$. For, if $a\in \text{Smooth}(X)$ then there are indexes $i_1,\ldots,i_n$,  where $n=\dim(X)$, such that $\{dx_{i_j}\}_{j=1,\ldots, n}$ are linearly independent in $T_{X,a}$. If we denote $\pi:\C^m\rightarrow \C^n$ the $(i_1,\ldots,i_n)$-projection, then there is an open ball $B$ centered at $a$ with $\pi(B\cap X)$ open in $\C^n$ and an analytic section $s: \pi(B\cap X)\rightarrow X\subseteq \C^{m}$ of $\pi$. Moreover, since $U_a:=X\cap B$ is semialgebraic and $s$ is the inverse of the semialgebraic map $\pi_a:=\pi|_{U_a}$, we  get that also $s$ is  semialgebraic. Finally, the collection of charts $(U_a,\pi_a)$ for $a\in \text{Smooth}(X)$ is an atlas of a complex analytic manifold (see e.g. \cite[Cor.\,1.26]{Mumford} for details).

In our case, given a complex algebraic group $G$, let $\text{Smooth}(G)$ be the set of all preimages of the smooth points by the Zariski charts. Then, clearly $\text{Smooth}(G)=G$. 
Indeed, given $g\in G$ and $h\in \text{Smooth}(G)$, we consider the map $G\rightarrow G:x\mapsto gh^{-1}x$.
Since this map is a rational isomorphism in the Zariski charts, we get that the image of $h$ is also a smooth point.
In particular, $G$ has a canonical 
$n$-dimensional complex manifold structure, for some $n$. {\it Henceforth, when we refer to open subsets of $G$ we mean open with respect to this complex manifold structure}.

\begin{lemma}\label{C-Nash structure} 
Let $G$ be a $n$-dimensional complex algebraic group. 
Then, there exists a $\mathbb{C}$-Nash atlas $\mathcal{A}$ such that $G$  equipped with $\mathcal{A}$ is a locally $\mathbb{C}$-Nash group. 
Moreover, if $(Y,f)$ is a Zariski chart of the identity of $ G$, then there exist a projection 
$\pi:\mathbb{C}^m\rightarrow \mathbb{C}^{n}$ on some coordinates and an open subset $U\sub Y$ such that $\mathcal{A}=\mathcal{A}_{(U,\pi  f|_U)}$. 
\end{lemma}
\begin{proof}
By the paragraph above, we already have an analytic structure on $G$, so it is enough to check properties \emph{(i)} and \emph{(ii)} of Fact\,\ref{compatibility0}.
Let $f:Y\rightarrow X\sub \mathbb{C}^m$ be a Zariski chart of the identity $e\in Y \sub G$.
Let $U\sub Y$ be an open neighbourhood of $e$ for which there exists a projection $\pi:\mathbb{C}^m\rightarrow \mathbb{C}^{n}$ on some coordinates such that 
$(U,\pi f|_{U})$ is a chart of the analytic structure of $G$. 
Note that we can assume that $f(U)$ is a semialgebraic set, so that $V:=(\pi f)(U)$ is also semialgebraic. 
In particular, the inverse $s:V \rightarrow f(U)\sub \mathbb{C}^m$ of the projection $\pi:f(U)\rightarrow V$ is semialgebraic and analytic, {\it i.e.}, a $\mathbb{C}$-Nash map. 

Since $G$ is an algebraic group, there exists an open neighbourhood  $U'\sub U$ of $e$ such that
\[
fm (f^{-1},f^{-1}):f(U')\times f(U')\rightarrow f(U)\sub \mathbb{C}^m:(f(y_1),f(y_2)) \mapsto f(y_1y_2) 
\]
satisfies that each coordinate function is the quotient of two polynomials whose denominator does not vanish, where $m:G\to G:(g,h)\mapsto g\cdot h$. 
Without loss of generality, we can assume that $f(U')$ is semialgebraic, so the map above is a $\mathbb{C}$-Nash map. 
Finally, the composition
\[
\phi \,m  (\phi^{-1},\phi^{-1})|_{ \phi(U')\times \phi(U')}=\pi \big(fm (f^{-1},f^{-1})\big) (s,s)|_{ \phi(U')\times \phi(U')}, 
\]
where $\phi:=\pi f|_{U}$, is a $\mathbb{C}$-Nash map, as required.

Similarly, one shows that for each $g\in G$ there exists $U_g\sub U$ such that $\phi \circ -^{g} \circ \phi^{-1}|_{\phi(U_g)}$ is a $\mathbb{C}$-Nash map making use of the fact  that conjugation by $g$ is a rational map in the Zariski charts. 
\end{proof}
\begin{remark}\label{rmk:algGRNash}1) Let us note that in the above proof we can endow $G$ with a \em finite \em $\C$-Nash atlas $\mathcal{A}$, i.e., the group $G$ is a \emph{$\C$-Nash group}. Let $f:Y\rightarrow X\subseteq \C^m$ be a Zariski chart of $G$. We can assume that $X$ is irreducible and therefore connected. Fix  $i_1,\ldots,i_n$ in $\{1,\dots, m\}$ and let $U$ be the Zariski open subset of $X$ of points $a\in X$ such that $\{dx_{i_j}\}_{j=1,\ldots, n}$ are linearly independent in $T_{X,a}$. Clearly, it suffices to prove that $U$ is the union of finitely finitely many $\C$-Nash charts. Note that $U$ is again connected because its complement is a Zariski closed set of the irreducible $X$. Let $\pi:\C^m\rightarrow \C^n$ by the $(i_1,\ldots,i_n)$-projection. Then, by what we wrote just before  Lemma \ref{C-Nash structure}, the continuous semialgebraic map  $\pi|_U:U\rightarrow \pi(U)$ is a covering map. In parti\-cular, it must be a finite $d$-sheeted covering for some $d\in \N$. Since $\pi(U)$ is a semialgebraic set, there are finitely many semialgebraic simply-connected open subsets $U_1,\ldots,U_\ell$ of $\pi(U)$ whose union is $\pi(U)$ (see \cite{Wilkie05}). Then, each $V_i:=\pi^{-1}(U_i)$ is an open semialgebraic subset of $U$ and we can consider the restriction $\pi|_{V_i}:V_i \rightarrow U_i$, which is again a finite cover. Since $U_i$ is simply-connected, there are open  semialgebraic disjoint  subsets $V_{i,1},\ldots, V_{i,d}$ of $V_i$ whose union is $V_i$ and such that $\pi_{ij}:=\pi|_{V_{i,j}}:V_{i,j}\rightarrow U_i$ is a homeomorphism. Note that $U=\bigcup_{i,j} V_{ij}$ and each pair $(V_{ij},\pi_{ij})$ is a chart of the complex manifold structure of $G$, as required.

\noindent 2) If the irreducible algebraic group $G$ in Lemma \ref{C-Nash structure} is defined over $\R$, then the set of real points $G(\R)$ of $G$ is a real algebraic group. $G(\R)$ has a canonical structure of Nash group which naturally comes from the complex one.
 Indeed, in the proof of Lemma \ref{C-Nash structure} we can pick $U$ such that $f(U)$ is an open subset of $X$ invariant with respect to complex conjugation. Since $\overline{\pi(\overline{x})}=\pi(x)$ for all $x\in f(U)$, we clearly have that the section $s:V\rightarrow f(U)$ must satisfy $\overline{s(\overline{x})}=s(x)$ for all $x\in V$, where $V=\pi f(U)$. Therefore if we set
$U(\R):=f^{-1}(f(U)\cap \R^m)$ then we get that the chart $(U(\R),\pi f|_{U(\R)})$ induces a Nash group structure on the set of
 real points $G(\R)$ of $G$. 

Moreover, recall that by Chevalley's theorem \cite[Lem.\,1 \& Thm.\,16]{Rosenlicht}, abstract algebraic groups are open Zariski subsets of algebraic subsets of a projective space $\mathbb{P}^m(\C)$ for some $m\in \N$. If $G$ is defined over $\R$ then the set of real points $G(\R)$ of $G$ is an open Zariski subset of an algebraic subset of the real projective space $\mathbb{P}^m(\R)$. On the other hand, the real projective space $\mathbb{P}^m(\R)$ is biregularly isomorphic to an algebraic subset of $\R^{\ell}$ for some $\ell\in \N$ (see \cite[Thm.\,3.4.4]{Bochnak_Coste_Roy}). Therefore, there is a Nash embedding of the connected component $G(\R)^0$ of $G$ into $\R^\ell$, that is, the Nash group $G(\R)^0$ is \emph{Nash affine}.

In affine Nash  manifolds there are no infinite strictly descending chains of Nash subsets, where a Nash subset is the common zero set of finitely many Nash functions (see \cite[\S 2.8]{Bochnak_Coste_Roy}). In particular, we can consider the \emph{Nash closure} of any subset of an affine Nash group.
\end{remark}

We now analyze locally $\mathbb{C}$-Nash maps between algebraic groups. The category of algebraic groups is not a full subcategory of the category of analytic groups, we will show in Corollary \ref{ALGisomorfismo} that it is indeed a full subcategory of the category of locally $\C$-Nash groups. Recall that given topological groups $G_1$ and $G_2$, a \emph{local isomorphism} is a homeomorphism $f:W\rightarrow f(W)$ between open subsets $W$ and $f(W)$ of $G_1$ and $G_2$ respectively, such that $f(xy)=f(x)f(y)$ for all $x,y\in W$ with $xy\in W$.

\begin{theorem}\label{ALGisogenia}
Let $G_1$ and $G_2$ be irreducible complex algebraic groups. 
If there exist an open semialgebraic subset $W$ of the identity of $G_1$ and a locally $\mathbb{C}$-Nash map $g:W\rightarrow G_2$ that is a local isomorphism, then $G_1$ and $G_2$ are biregularly isogenous.
\end{theorem}

\begin{proof}In the first part of the proof we will assume that $g:W\rightarrow G_2$ is just a locally $\mathbb{C}$-Nash local homomorphism, that is, satisfies that $g(xy)=g(x)g(y)$ for all $x,y\in W$ such that $xy\in W$. 

For each $i=1,2$, let $n_i$ denote the dimension of $G_i$ and let $f_i:Y_i\rightarrow X_i$ be a Zariski chart of the identity of $G_i$, where $Y_i \sub G_i$ and $X_i\sub \mathbb{C}^{m_i}$, and $f_i$ sends the identity to $0$. Let $V_i$ be an open (euclidean) subset of $X_i$ containing $0$ for which there exist projections $\pi_i:V_i \rightarrow \mathbb{C}^{n_i}$ over some coordinates such that $(f_i^{-1}(V_i),\pi_i\circ f_i)$ is a chart of the locally $\mathbb{C}$-Nash structure of $G_i$. 
We can assume that the projections are over the first $n_i$ coordinates. 

Note that, shrinking $W$, we can also assume that $g(W)\sub Y_2$. 
Moreover, we can assume that the open subset $U:=(\pi_1\circ f_1)(W)$ of $\mathbb{C}^{n_1}$ is contained in $\pi_1(V_1)$ and, therefore, 
$(\pi_2\circ f_2)\circ g\circ (\pi_1 \circ f_1)^{-1}|_U=\pi_2\circ (f_2 \circ g \circ f_1^{-1})\circ \pi_1^{-1}|_U$ is a $\mathbb{C}$-Nash map.

In other words, if we denote $\widetilde{g}=f_2 \circ g \circ f_1^{-1}|_{\pi_1^{-1}(U)}=(\widetilde{g}_1,\ldots,\widetilde{g}_{m_2})$ 
then there exist polynomials $Q_j\in \mathbb{C}[\tt{z}_1,\ldots ,\tt{z}_{n_1},\tt{y}]$, $Q_j\neq 0$, such that $$Q_j(z_1,\ldots,z_{n_1},\widetilde{g}_j(\pi_1^{-1}(z_1,\ldots,z_{n_1})))=0$$ for $(z_1,\ldots,z_{n_1})\in U$ 
and $j=1,\ldots,m_2$. 
So, if we consider
\[
Z_j=\{(z_1,\ldots,z_{m_1},y_1,\ldots,y_{m_2}) \suchthat Q_j(z_1,\ldots,z_{n_1},y_j)=0\}\sub \mathbb{C}^{m_1}  \times \mathbb{C}^{m_2} 
\]
and the Zariski closed subset $Z:= (X_1\times X_2)\cap \bigcap^n_{j=1} Z_{j}$ of $X_1\times X_2$ then
the graph of $\widetilde{g}$ is contained in $Z$. 
Furthermore, note that $Z$ is $n_1$-dimensional near $0$. 
Now, since the graph of $\widetilde{g}$ is an analytic manifold, we also deduce that $Z$ must be irreducible.

All in all, we have shown that $A_W:=\text{Graph}(g|_{W})\sub G_1\times G_2$ is contained in its Zariski closure $Z_W$ in $Y_1\times Y_2$, 
which is a closed Zariski irreducible subset of $Y_1\times Y_2 \sub G_1\times G_2$ and $n_1$-dimensional near $0$.
Denote by $B_W:=\overline{A_W}^\text{Zar}$ the Zariski closure of $A_W$ in $G_1\times G_2$. 
Let us show that $B_W=\overline{Z_W}^\text{Zar}$. 

Indeed, as $A_W\sub Z_W$ we have $B_W\sub \overline{Z_W}^\text{Zar}$. 
On the other hand, $A_W$ is contained in the Zariski closed subset $B_W\cap (Y_1\times Y_2)$ of $Y_1\times Y_2$ and, therefore, $Z_W$ is also contained, 
so $\overline{Z_W}^\text{Zar}\sub B_W$. 
Note that $B_W$ is irreducible of dimension $n_1$ near $0$ because the closure of an irreducible set is 
irreducible and $B_W\cap (Y_1\times Y_2)=\overline{Z_W}^\text{Zar}\cap (Y_1\times Y_2)=Z_W$.

Now, let $\mathcal{F}$ be the family of open (in the Euclidean topology) semialgebraic subsets 
$W_1$ of $W$  such that $W_1W_1 \sub W$ and $W_1=W_1^{-1}$. 
For each $W_1\in \mathcal{F}$, we consider the Zariski closure $B_{W_1}$ of 
$A_{W_1}:=\text{Graph}(g|_{W_1})\sub G_1\times G_2$ in $G_1\times G_2$, which is again irreducible and 
$n_1$-dimensional.
Let us denote $B=\bigcap_{W_1\in \mathcal{F}}B_{W_1}$. 
This intersection is finite, so $B=B_{W_1}$ for some $W_1\in \mathcal{F}$.

Let us see that $B$ is a subgroup of $G_1\times G_2$. 
Let $W_2\in \mathcal{F}$, with $W_2W_2\sub W_1$, and note that $B_{W_2}=B_{W_1}=B$. 
Pick a point $a\in A_{W_2}$ and consider the Zariski closed subset
\[
L_a=\{x\in G_1\times G_2 \suchthat ax \in B\}.
\]
Since $aA_{W_2}\sub A_{W_2}A_{W_2} \sub A_{W_1} \sub B$, we get $A_{W_2} \sub L_a$ and, therefore, $B\sub L_a$. 
On the other hand, consider the Zariski closed subset
\[
R=\{a\in G_1\times G_2 \suchthat aB \sub B\}. 
\]
We have showed that $A_{W_2}\sub R$, so that $B\sub R$. 
Finally, let us show that $B^{-1}=B$. 
Take the Zariski closed set $I=\{a\in G_1\times G_2 \suchthat a^{-1}\in B \}$. 
Since $W_2^{-1}=W_2$, we have that $A_{W_2}\sub I$ and, therefore, $B\sub I$, as required.

We have proved that $B$ is an irreducible algebraic subgroup of $G_1\times G_2$ of dimension $n_1$. 
As $B$ is a group, it is connected and pure dimensional. 

\smallskip
In the second part of the proof, we assume that $g:W\rightarrow G_2$ is actually a $\C$-Nash local isomorphism. In particular, note that $n:=n_1=n_2$ because $G_1$ and $G_2$ have the same dimension as complex analytic manifold and, therefore, as algebraic groups.
Consider the projection $p_1:B\rightarrow G_1$. Note that $p_1(B)$ is a $n$-dimensional subgroup of the 
irreducible $G_1$, 
so that $p_1$ is onto. 
We have that $\ker(p_1)=e\times F_2$, where $F_2$ is a subgroup of $G_2$. 
Moreover, since $\dim(B)=\dim(G_1)=n$, we deduce that $F_2$ is finite. 
Similarly, the projection $p_2:B\rightarrow G_2$ is onto and $\ker(p_2)=F_1\times e$, 
where $F_1$ is a finite subgroup of $G_1$. 
Therefore, $G_1/F_1=B/(F_1\times F_2)=G_2/F_2$, where the  equalities occur in the algebraic category. Hence, we conclude that $G_1$ and $G_2$ are isogenous.
\end{proof}

\begin{cor}\label{ALGisomorfismo}
Any locally $\mathbb{C}$-Nash homomorphism between irreducible complex algebraic groups is a biregular morphism.
\end{cor}
\begin{proof}
Let $G_1$ and $G_2$ be  irreducible complex  algebraic groups of dimension $n_1$ and $n_2$ respectively. Let $g:G_1\rightarrow G_2$ be a locally $\mathbb{C}$-Nash homomorphism. 
Let $W$ be an small enough neighbourhood of the identity of $G_1$ and consider the restriction $g|_W:W\rightarrow G_2$. 
Let $B$ be the $n_1$-dimensional irreducible subgroup of $G_1\times G_2$ obtained in the first part of the proof of Theorem\,\ref{ALGisogenia} (noting that in that proof we only assume that $g$ is a locally $\C$-Nash local homomorphism). By construction and without loss of generality, we can assume that $B$ is the Zariski closure of $A:=\text{Graph}(g|_{W})\sub G_1\times G_2$. Now, let $H$ be the graph of $g$, which is an analytic connected subgroup of $G_1\times G_2$ of dimension $n_1$. 
Since $H\cap B$ is an analytic subset of the analytic connected manifold $H$ which contains an open set 
(because it contains $A$), it follows that $H\cap B=H$, {\it i.e.}, $H$ is a subgroup of $B$. 
Since $B$ is a connected group of dimension $n_1$ and $H$ is a $n_1$-dimensional subgroup, it follows that $B=H$. 
In particular, the isomorphism $g$ is a morphism between algebraic varieties, as required.
\end{proof}

We finish this section analyzing the set of real points of algebraic groups defined over $\R$, which by Remark \ref{rmk:algGRNash} has a canonical structure of locally Nash group. We recall that if $G$ is an irreducible algebraic group defined over $\R$ then  -- being smooth --  $G(\R)$ is a Zariski-dense subset of $G$ of real dimension equal to $\dim(G)$ (see  \cite[p.\,8]{Silhol}).

\begin{prop}Let $G_1$ and $G_2$ be irreducible complex  algebraic groups defined over $\mathbb{R}$. For $i=1,2$, let $G_i(\mathbb{R})$, $G_i(\mathbb{R})^o$ and $\widetilde{G_i(\mathbb{R})^o}$ denote the set of real points, its connected component and the universal covering respectively. Then the following are equivalent:

\noindent \emph{(1)} There is a Nash local isomorphism between  $\widetilde{G_1(\mathbb{R})^o}$ and $\widetilde{G_2(\mathbb{R})^o}$.

\noindent  \emph{(2)} There is a biregular isogeny between  $G_1$ and $G_2$ defined over $\mathbb{R}$. 

\noindent  \emph{(3)} $G_1(\mathbb{R})^o$ and $G_2(\mathbb{R})^o$ are Nash isogenous.

\end{prop}
\begin{proof}(2) implies (1) is clear. We show (3) implies (2).   The graph of the isogeny is a Nash subgroup of  $G_1(\mathbb{R})^o \times G_2(\mathbb{R})^o$ of dimension $n:=\dim G_1(\mathbb{R})=\dim G_2(\mathbb{R})$. Take the Zariski closure $H_{\text{real}}$ of the latter in the real algebraic group $G_1(\R) \times G_2(\R)$, which is an algebraic subgroup  of dimension $n$. Then, pick the Zariski closure $H$ of $H_{\text{real}}$ in $G_1\times G_2$, which is again an algebraic subgroup of dimension $n$. The projection over  the irreducible $G_1$ is a constructible subgroup of $G_1$ of dimension $n$, so it equals $G_1$. Similarly for $G_2$. Therefore $H$ induces an biregular isogeny defined over $\mathbb{R}$.

(1) implies (3). By (1) we have a Nash local isomorphism $g:W\rightarrow G_2(\mathbb{R})^o$,  where  $W$ is an open connected semialgebraic subset of $G_1(\mathbb{R})^o$. Recall that $G_1(\mathbb{R})^o \times G_2(\mathbb{R})^o$ is  Nash affine group by Remark \ref{rmk:algGRNash}.
Let $\mathcal{F}$ be the family of open connected semialgebraic subsets
$W_1$ of $W$  such that $W_1W_1 \sub W$ and $W_1=W_1^{-1}$. For each $W_1\in \mathcal{F}$, we consider the Nash closure $B_{W_1}$ of
$\text{Graph}(g|_{W_1})\sub G_1(\mathbb{R})^o \times G_2(\mathbb{R})^o$, which is connected and $n$-dimensional.
Let us denote $B=\bigcap_{W_1\in \mathcal{F}}B_{W_1}$.
This intersection is finite, so $B=B_{W_1}$ for some $W_1\in \mathcal{F}$.

It is easy to show that  $B$ is a Nash connected subgroup of $G_1(\mathbb{R})^o \times G_2(\mathbb{R})^o$ of dimension $n$, which induces a Nash isogeny, as required.
\end{proof}

\section{Abelian locally $\K$-Nash groups}\label{ALKNgroups}

A simply connected abelian locally $\K$-Nash group is analytically isomorphic to $\K^n$, for some $n$. Thus, this section is dedicated to the study of the  local $\K$-Nash group structures on $\K^n$.  Since we want to consider both the real and complex case, we first recall some notions related to invariant meromorphic maps. Then,  we will introduce some notation which allows us to state the results in \cite{BDOAAT} -- on maps admitting an algebraic addition theorem--, and finally, we will apply these results to  our case.

We  recall that a meromorphic map $f:\mathbb{C}^n\dashrightarrow \mathbb{C}^m$ is an \emph{invariant meromorphic map} if 
$\overline{f(\overline{u})}=f(u)$ for each $u\in \mathbb{C}^n$ where $f$ is defined.  We use the symbol $\dashrightarrow$ to stress that these functions are not defined in all points of $\C^n$, they are defined off an analytic subset. We say that a map $f:\mathbb{C}^n\dashrightarrow \mathbb{C}^m$ is \emph{$\mathbb{C}$-meromorphic} 
if it is just meromorphic and \emph{$\mathbb{R}$-meromorphic} if it is invariant meromorphic.

Let  $\mathcal{O}_{\mathbb{K},n}$ be the ring of power series in $n$ variables with coefficients in $\mathbb{K}$ that 
are convergent in a neighbourhood of the origin,  and $\CM_{\mathbb{K},n}$ its quotient field.  As usual, by the identity principle for analytic functions, we identify $\mathcal{O}_{\mathbb{K},n}$ with the ring of germs of analytic 
functions at $0$, and $\CM_{\mathbb{K},n}$ with its quotient field. 

We say that  $\phi \in \CM_{\mathbb{K},n}^n$ admits an \emph{algebraic addition theorem} (AAT)  if 
$\phi _1,\ldots ,\phi _n$ are algebraically independent over $\mathbb{K}$ and $\phi(u+v)$ is algebraic over 
$\mathbb{K}(\phi _1(u),\ldots ,\phi _n(u), \phi _1(v),\ldots ,\phi _n(v))$, where $u$ and $v$ denote $n$-tuples of variables.

We recall that by the well-known description of abelian Lie groups, the only analytic structure on the  additive group  $\mathbb{K}^n$ is the standard one (the one given by the identity map, so that its compatible charts are exactly the analytic diffeomorphisms). Next result relates locally $\K$-Nash structures on $(\K,+)$ with maps admitting an algebraic addition theorem:

\begin{lemma}\label{AAT-star}
Let $(U,\phi )$ be a chart of the identity of the group $\mathbb{K}^n$ compatible with its standard analytic structure.
Then, the following are equivalent:

\noindent 
$(1)$There exists an open neighbourhood of the identity $U'\sub U$ such that 
$$\phi (U')\times \phi (U')\rightarrow \phi (U) :
(x,y)\mapsto \phi (\phi ^{-1}(x)+\phi ^{-1}(y))$$
is a $\mathbb{K}$-Nash map (and therefore, there exists an open neighbourhood $V\sub U$ of $0$ such that 
$(\mathbb{K}^n,+,\phi|_V)$ is a locally $\mathbb{K}$-Nash group), and 

\noindent
$(2)$ $\phi$ admits an AAT.
\end{lemma}
\begin{proof}
$(1)$ implies $(2)$:
By hypothesis, $\phi (U')$ is semialgebraic, since it is the projection of the domain of a semialgebraic map. 
Fix $i\in \{1,\ldots ,n\}$.
As we have mentioned in the definition of $\mathbb{K}$-Nash map, this hypothesis implies that there exists 
$P_i\in \mathbb{K}[X_1,\ldots ,X_{2n+1}]$, $P_i\neq 0$, such that
\[
P_i(x_1,\ldots ,x_n,y_1,\ldots ,y_n,\phi _i(\phi ^{-1}(x) + \phi ^{-1}(y)))\equiv 0 \text{ on } \phi (U')\times \phi (U'),
\]
where $x:=(x_1,\ldots ,x_n)$ and $y:=(y_1,\ldots ,y_n)$.
Since $\phi$ is a diffeomorphism, letting $u:=\phi ^{-1} (x)$ and $v:=\phi ^{-1}(y)$, we deduce that 
\[
P_i(\phi _1(u),\ldots ,\phi _n(u),\phi _1(v),\ldots ,\phi _n(v),\phi _i(u+v))\equiv 0 \text{ on } U'\times U'.
\]
In addition, the coordinate functions $\phi _1,\ldots ,\phi _n$ are clearly algebraically independent.
So $\phi$ admits an AAT.

\smallskip

$(2)$ implies $(1)$:
Fix $i\in \{1,\ldots ,n\}$.
If $\phi$ admits an AAT then there exists $P_i\in \mathbb{K}[X_1,\ldots ,X_{2n+1}]$, $P_i\neq 0$, such that
\[
P_i(\phi _1(u),\ldots ,\phi _n(u),\phi _1(v),\ldots ,\phi _n(v),\phi _i(u+v))\equiv 0 \text{ on } U'\times U'
\]
for some open neighbourhood of the identity $U'\sub U$.
Since $\phi$ is a diffeomorphism, we can let $x:=\phi (u)$ and $y:=\phi (v)$ and  shrinking $U'$, if  necessary,  to make it 
semialgebraic, as required.
\end{proof}

The main purpose of this section is to improve the above lemma. Namely, we want to prove that any locally $\K$-Nash structure on $(\K^n,+)$ comes from a \emph{global} $\K$-meromorphic map $f:\C^n\dashrightarrow \C^n$ admitting an AAT. For that reason we introduce the following notation. Let $f:\mathbb{C}^n\dashrightarrow \mathbb{C}^n$ be  a $\mathbb{K}$-meromorphic map admitting an AAT and  satisfying 
 the following condition: 
\begin{itemize}
  \item[(*)] there exist $a\in \mathbb{K}^n$ and an open neighbourhood $U\sub \mathbb{K}^n$ of $0$ such that
\[
\varphi : U\rightarrow \mathbb{K}^n: u\mapsto \varphi(u):=f(u+a) 
\]
is an analytic diffeomorphism onto its image. 
\end{itemize} 
By \cite[Lem.\,3]{BDOAAT}  such $\varphi$  is algebraic over $f$, and by \cite[Lem.\,4]{BDOAAT}  $\varphi$ admits an AAT. We will denote by $(\mathbb{K}^n,+,f)$  the locally $\K$-Nash group $(\mathbb{K}^n,+,\varphi|_V)$, where  $V\sub U$  is an open neighbourhood of $0$ given by Lemma\,\ref{AAT-star}.

To make the notation sound, it remains to check that the locally $\mathbb{K}$-Nash group structure is independent of $a$ and the domains $U$ and $V$. That is, we 
have to show that given a $\mathbb{K}$-meromorphic map $f:\mathbb{C}^n\dashrightarrow \mathbb{C}^n$ admitting an AAT, and given 
$a_1,a_2\in \mathbb{K}^n$ such that
\[
\varphi_1: U_1\rightarrow \mathbb{K}^n: u\mapsto f(u+a_1), \quad \varphi_2: U_2\rightarrow \mathbb{K}^n: u\mapsto f(u+a_2),
\]
are analytic diffeomorphisms onto their corresponding images, we have that $(\mathbb{K}^n,+,\varphi_1|_{V_1})$ and $(\mathbb{K}^n,+,\varphi_2|_{V_2})$ 
are isomorphic as locally $\mathbb{K}$-Nash groups, where $V_1\sub U_1$ and $V_2\sub U_2$ are given by Lemma\,\ref{AAT-star}.
By \cite[Lem.\,3]{BDOAAT}  $\varphi_1$ is algebraic over $\mathbb{K}(\varphi_2)$.
Hence, by Proposition\,\ref{compatibilityAut} the identity map is a locally $\mathbb{K}$-Nash isomorphism between
$(\mathbb{K}^n,+,\varphi_1|_{V_1})$ and $(\mathbb{K}^n,+,\varphi_2|_{V_2})$.

\medskip

Hereafter, when we write \emph{$(\mathbb{K}^n,+,f)$ is a locally $\mathbb{K}$-Nash group}, 
we are also assuming that $f:\C^n\dashrightarrow \C^n$ is a $\K$-meromorphic map  admitting an AAT and satisfying condition (*) above.

\begin{rem}\label{LNG-LCNG}
If $(\mathbb{R}^n,+,f)$ is a locally Nash group then $(\mathbb{C}^n,+,f)$ is a locally $\mathbb{C}$-Nash group.
Indeed, by definition  the $\R$-meromorphic  $f:\C^n\dashrightarrow\C^n$ admits an AAT and satisfies condition $(*)$ for $\mathbb{K}=\mathbb{R}$. That is, there exist $a\in \mathbb{R}^n$ and an open neighbourhood $U\sub \mathbb{R}^n$ of $0$ such that the map
$U\rightarrow \mathbb{R}^n: u\mapsto f(u+a) 
$
is an analytic diffeomorphism.  Hence, the determinant of its Jacobian at 0 is not zero. Thus, the determinant of the  Jacobian at 0 of $f(u+a):\C^n\dashrightarrow \C^n$ is not zero, and we get that 
 $f$ is a local diffeomorphism at $0$.  Conversely,  if $(\mathbb{C}^n,+,f)$ is a locally $\mathbb{C}$-Nash group  and $f$ is $\R$-meromorphic then $(\mathbb{R}^n,+,f)$ is a locally Nash group. For, working again  with the Jacobian it is easy to find $a\in\R^n$ such that condition (*) is satisfied. 
\end{rem}

Now we are ready to prove the main results of this section. We will make use of the following extension result from \cite{BDOAAT}:

\begin{fact}\label{T2}{\cite[Thm.\,1]{BDOAAT}}
Let $\phi :=(\phi _1,\ldots ,\phi _n)\in \CM_{\mathbb{K},n}^n$ admit an AAT.
Then, there exist $\psi :=(\psi _1,\ldots ,\psi _n)\in \CM_{\mathbb{K},n}^n$ admitting an AAT and  algebraic over $\mathbb{K}(\phi)$, 
and an additional meromorphic series $\psi _0\in \CM_{\mathbb{K},n}$ algebraic over $\mathbb{K}(\psi )$ such that:
\begin{enumerate}
\item[$(1)$] For each $f(u)\in \mathbb{K}\big(\psi _0(u),\ldots ,\psi _n(u)\big)$,
\smallskip
\begin{enumerate}
\item[$(a)$] $f(u+v)\in \mathbb{K}\big(\psi _0(u),\ldots ,\psi _n(u),\psi _0(v),\ldots ,\psi _n(v)\big)$ and
\smallskip
\item[$(b)$] $f(-u)\in \mathbb{K}\big(\psi _0(u),\ldots ,\psi _n(u)\big)$.
\end{enumerate}
\smallskip
\item[$(2)$] Each $\psi _0,\ldots ,\psi _n$ is the quotient of two convergent power series whose complex domain of convergence is 
$\mathbb{C}^n$.
\end{enumerate}
\end{fact}

\begin{theorem}\label{T1}
Every simply connected $n$-dimensional abelian locally $\mathbb{K}$-Nash group is isomorphic to some
$(\mathbb{K}^n,+, f )$.
\end{theorem}
\begin{proof}
Let $G$ be a simply connected $n$-dimensional abelian locally $\mathbb{K}$-Nash group equipped with a Nash atlas 
$\mathcal{B}:=\{(W_i,\ff _i)\}_{i\in I}$. 
In particular, $G$ is an analytic group with this atlas and, therefore, there exists an isomorphism of analytic groups
$
\alpha :G \rightarrow \mathbb{K}^n,
$
where $\mathbb{K}^n$ is equipped with its unique analytic group structure, the standard one.
Since $\mathcal{B}$ is a $\mathbb{K}$-Nash atlas for $G$, we have that
\[
\mathcal{A}:=\{ \alpha(W_i), \ff _i  \alpha^{-1}\}_{i\in I}
\]
is a $\mathbb{K}$-Nash atlas for $\mathbb{K}^n$ compatible with its standard analytic structure.
Moreover, $G$ equipped with $\mathcal{B}$ is clearly locally $\mathbb{K}$-Nash isomorphic to 
$\mathbb{K}^n$ equipped with $\mathcal{A}$ (see Proposition\,\ref{compatibilityAut}).

Now, consider a chart of the identity $(U,\ff)\in \mathcal{A}$.
Firstly, note that, as analytic chart, $(U,\ff)$ must be compatible with the standard analytic structure of 
$\mathbb{K}^n$, so $\ff$ is an analytic diffeomorphism.
Also, being a chart of a locally $\mathbb{K}$-Nash group structure, it satisfies condition $(1)$ of Lemma\,\ref{AAT-star}, so
$\ff$ admits an AAT.

Next, we apply Fact\,\ref{T2} to obtain a $\K$-meromorphic map
$$\psi:=(\psi_1,\ldots,\psi_n):\C^n \dashrightarrow \C^n$$
admitting an AAT and such that $\psi$ is algebraic over $\K(\varphi)$ as power series. By algebraic independence, it follows that $\varphi$ is also algebraic over $\K(\psi)$.

We will prove that $(\K^n, +, \psi)$ is a locally $\K$-Nash group isomorphic to $(\K^n, +, \ff|_{U})$ and hence to $\K^n$ with its atlas $ \mathcal{A}$, so we can take $\psi$  for the $f$ in the statement. The only thing which remains to be proved is that $\psi$ satisfies condition (*). Indeed,  there exists an open dense subset 
$W\sub \mathbb{K}^n$ such that 
$
\psi|_W:W\rightarrow \mathbb{K}^n
$
is analytic.
Let $\Delta$ denote the determinant of the Jacobian of $\psi|_W$. Since $\ff$ is an analytic diffeomorphism, we deduce that
$
{d\ff} $ at 0
is a vector space  isomorphism. 
Hence, $ {d\ff} _1,\ldots , {d\ff} _n$ are linearly independent over $\CM_{\mathbb{K},n}$, and so they are  $d\psi _1,\ldots ,d\psi _n$.
Thus, 
$\Delta$ is not identically zero on $W$, so there exists $a\in  W$ such that $\Delta (a)\neq 0$. 
Thus, $\psi$ is a local diffeomorphism at $a$, so there exists an open neighbourhood $V\sub \mathbb{K}^n$ of $0$ such that
$V\rightarrow \mathbb{K}^n: u\mapsto \psi(u+a)$ is a diffeomorphism onto its image, as required. 
\end{proof}

The following result characterises  isomorphism between groups of the above form.

\begin{prop}\label{compatibilityAut.1}
 The locally $\mathbb{K}$-Nash groups $(\mathbb{K}^n,+, f)$ and $(\mathbb{K}^n,+, g)$ are isomorphic if and only if 
there exists $\alpha \in \GL_n(\mathbb{K})$ such that $g  \alpha$ is algebraic over $\mathbb{K}(f)$.
\end{prop}
\proof
By definition, there exist $a, b\in \mathbb{K}^n$ and an open neighbourhoods  $U, V\sub\K^n$ of 0 such that the maps
$$
f_a:U \rightarrow \mathbb{K}^n : u\mapsto f(u+a)
\text{ and }
g_b : V\rightarrow \mathbb{K}^n : u\mapsto g(u+b)
$$are  charts  of the corresponding groups.  Note that any analytic automorphism of $(\K,+)$ is of the form $\alpha\in \GL_n(\K)$. Note that any analytic automorphism of $(\K,+)$ is of the form $\alpha\in\GL_n(\K)$. Thus,  by Proposition\,\ref{compatibilityAut} it suffices to prove that 
for any $\alpha \in \GL_n(\mathbb{K})$,  $g  \alpha$ is algebraic over $\mathbb{K}(f)$ if and only if 
$g_b  \alpha$ is algebraic over $\mathbb{K}(f_a)$. The latter is true by \cite[Lem.\,3]{BDOAAT}.
\endproof

Next, we introduce an  invariant   of isomorphism types of abelian  locally $\K$-Nash groups, which will be specially useful for the classification of the latter groups in dimension one and two (see \cite{BDO2ALKNG}).

 Let $\Lambda$ be a discrete subgroup of $\mathbb{C}^n$.
Then, there exist $r\leq 2n$ and $\lambda _1,\ldots ,\lambda _r\in \Lambda$, linearly independent over $\mathbb{R}$, such that 
\[
\Lambda = \mathbb{Z}\lambda _1 \oplus \ldots \oplus \mathbb{Z}\lambda _r.
\]
We call $r$ the rank of $\Lambda$ (the dimension of $\Lambda$ as a free $\mathbb{Z}$-module), and denote it  by $\rank \Lambda$.
A discrete subgroup $\Lambda$ of $\mathbb{C}^n$ is a \emph{lattice} if $\rank  \Lambda =2n$.
We say that a subgroup $G\leq\mathbb{C}^n$ is an \emph{invariant subgroup} if $G =\overline{G}$, {\it i.e.}, if $g\in G$ implies that $\overline{g}\in G$.

The previous concepts are related to meromorphic maps as follows.
Given a meromorphic map $f:\mathbb{C}^n\dashrightarrow \mathbb{C}^m$, we define the \emph{group of periods of $f$} as
\[
\Lambda _f:=\{  a \in \mathbb{C}^n \suchthat f(u)=f(u+a)\},
\]
where $f(u)=f(u+a)$ means that if $f=g/h$ then $g(b)h(b+a)=h(b)g(b+a)$, for all $b\in \mathbb{C}^n$.
Note that $\Lambda _f$ is a subgroup of $\mathbb{C}^n$, which is invariant, provided that $f$ is also invariant. We have the following:

\begin{lemma}\label{discrete groups}
Let $f:\mathbb{C}^n\dashrightarrow \mathbb{C}^n$ be a meromorphic map.
We have:
\\ $(1)$ If $f$ is a local diffeomorphism at some point, then $\Lambda _f$ is a discrete subgroup of  $\mathbb{C}^n$.
\\ $(2)$ If $f$ is a $\K$-meromorphic map, $a\in \mathbb{\K}^n$ and, for some open neighbourhood of $0$,
$U\sub \mathbb{K}^n$, the restriction of $f(u+a)$ to $U$ is an analytic diffeomorphism, then $\Lambda _f$ 
is a discrete subgroup of $\mathbb{C}^n$. 
\\ $(3)$ If $\Lambda _f$ is a discrete subgroup of $\mathbb{C}^n$ and $\alpha \in \GL_n(\mathbb{C})$, then 
$\Lambda _{f  \alpha}$ is a discrete subgroup of $\mathbb{C}^n$ with $\rank  \Lambda _{f  \alpha}=\rank  \Lambda _f$.
\end{lemma}
\begin{proof}
$(1)$ We can assume that $f$ is a local diffeomorphism at 0.
If $\Lambda _f$ is not a discrete subgroup of $\C^n$
then, there exists an infinite sequence $\{ a_k \suchthat k\in \mathbb{N}\}$ of points of $\Lambda _f$ that converges to some 
$a\in \mathbb{C}^n$. Since  $f$ is a local diffeomorphism at $0$, we can take $\epsilon>0$ such that $f$ is injective and analytic on an open ball of radius $\epsilon$ centered at $0$. 
Let $N\in \mathbb{N}$ be such that $\| a_k-a_N\| <\epsilon$ for all $k\geq N$.
Now,  $\Lambda _f$ is a subgroup of $\mathbb{C}^n$,  thus  $a_k-a_N\in \Lambda _f$ for all $k\in \mathbb{N}$.
This implies that $f(a_k-a_N)=f(0)$ for all $k\in \mathbb{N}$, which contradicts that $f$ is injective in the mentioned ball.

$(2)$  By (1) we have the case $\K=\C$, the case $\K=\R$ reduces to the former one  by the argument in Remark\,\ref{LNG-LCNG}.

$(3)$ Take $r\leq 2n$ and $\lambda _1,\ldots ,\lambda _r\in \Lambda$ linearly independent over $\mathbb{R}$ such that 
$\Lambda _f= \mathbb{Z}\lambda _1 \oplus \ldots \oplus \mathbb{Z}\lambda _r$.
Then, $\alpha ^{-1}(\lambda _1),\ldots ,\alpha ^{-1}(\lambda _r)\in \Lambda$ are linearly independent over $\mathbb{R}$ and
$\Lambda _{f  \alpha}= \mathbb{Z}\alpha ^{-1}(\lambda _1) \oplus \ldots \oplus \mathbb{Z}\alpha ^{-1}(\lambda _r)$.
\end{proof}

\begin{lemma}\label{algebraicity of periods} Let $f,g:\mathbb{C}^n\dashrightarrow \mathbb{C}^n$ be meromorphic maps such that  $\Lambda _g$ is a discrete subgroup of $\mathbb{C}^n$ and  $g$ is algebraic over $\mathbb{C}(f)$.  Then, $\Lambda _f$ is also discrete and there exists $N\in \mathbb{N}\setminus \{0\}$ such that $N\Lambda _f\leq\Lambda _g$,  
and so  $\rank \Lambda _f \leq \rank \Lambda _g$.  Moreover, if the coordinate functions of $g$ are algebraically independent over $\mathbb{C}$ then $\rank \Lambda _f = \rank \Lambda _g$.
\end{lemma}
\begin{proof}We begin by proving that $\Lambda_f$ is discrete. Fix $j\in \{1,\ldots ,n\}$ and let $P_j(Z)$ be the minimum polynomial of $g_j(u)$ over $\mathbb{C}(f(u))$. For every $\lambda \in \Lambda _f$ we have that  $g_j(u+\lambda)$ is also a root of $P_j(Z)$. Therefore there are $\lambda_1,\ldots,\lambda_{\ell_j} \in \Lambda _f$ such that for every $\lambda\in \Lambda _f$ there is $i\in \{1,\ldots, \ell_j\}$ with $g_j(u+\lambda)=g_j(u+\lambda_i)$, and so $\lambda-\lambda_i \in \Lambda_{g_j}$. Thus, $(\Lambda_f+\Lambda_{g_j})/\Lambda_{g_j}$ has finite order $\ell_j\in \N^*$ and we can conclude since the order of $(\Lambda_f+\bigcap^n_{j=1} \Lambda_{g_j})/\bigcap^n_{j=1} \Lambda_{g_j}$ is less or equal than $\ell_1\cdots \ell_n$.

Now, we show that there exists $N\in \mathbb{N}\setminus \{0\}$ such that $N\Lambda _f\leq\Lambda _g$. We may assume that $\Lambda _f\neq \{0\}$. Take $\lambda \in \Lambda _f\setminus \{0\}$ and fix $j\in \{1,\ldots ,n\}$.
Since $g_j(u+k \lambda)$ is a root of $P_j(Z)$ for each $k\in \mathbb{Z}$, there exist $k_1,k_2\in \mathbb{Z}$, $k_2>k_1$, such that 
$g_j(u+k_1\lambda )=g_j(u+k_2\lambda)$. Define $N_j:=k_2-k_1\in \mathbb{N}\setminus \{0\}$, and consider $N_\lambda$ the l.c.m. of $N_1,\ldots ,N_n$. Let $\{\lambda_1,\ldots ,\lambda_m\}$ be a basis for $\Lambda _f$, and denote by $N$ the l.c.m. of the $N_{\lambda_1},\ldots, N_{\lambda_m}$. Clearly $N\Lambda _f\leq\Lambda _g$, as required.
This also shows that $\Lambda _g$ contains at least $\rank \Lambda _f$ linearly independent vectors over $\mathbb{R}$, so 
$\rank\Lambda _f \leq \rank \Lambda _g$.
The other assertion follows by symmetry, since if $g_1,\ldots ,g_n$ are algebraically independent over $\mathbb{C}$
then $f$ is algebraic over $\mathbb{C}(g)$.
\end{proof}

Next  result states that the rank is an invariant of the isomorphism class of a locally $\K$-Nash group. However, we point out that by the classification of two-dimensional locally $\K$-Nash groups provided in \cite{BDO2ALKNG}, the rank does not characterise the isomorphism class.

\begin{proposition}\label{different ranks}
\ Let $(\mathbb{K}^n,+,f)$ and $(\mathbb{K}^n,+,g)$ be isomorphic locally $\mathbb{K}$-Nash groups, for some $\mathbb{K}$-meromorphic maps 
$f,g:\mathbb{C}^n\dashrightarrow \mathbb{C}^n$ that admit an AAT.
If $(\mathbb{K}^n,+,f)$ and $(\mathbb{K}^n,+,g)$ are isomorphic then $\rank \Lambda _f=\rank \Lambda _g$.
\end{proposition}
\begin{proof}
By Proposition\,\ref{compatibilityAut.1}, there exists $\alpha \in \GL_n(\mathbb{K})$ such that $g  \alpha$ is algebraic 
over $\mathbb{K}(f)$.
We note that, by Lemma\,\ref{discrete groups}.$(4)$ and $(5)$, both $\Lambda _g$ and $\Lambda _f$ are discrete subgroups of 
$\mathbb{C}^n$.
By Lemma\,\ref{discrete groups}.$(6)$, $\Lambda _{g  \alpha }$ is also a discrete subgroup of $\mathbb{C}^n$,
with $\rank \Lambda _{g  \alpha}=\rank \Lambda _g$.
Now, by Lemma\,\ref{algebraicity of periods}, $\rank \Lambda _{g  \alpha}=\rank \Lambda _f$, so
$\rank \Lambda _f=\rank \Lambda _g$.
\end{proof}

Now, we show that the classification of connected abelian locally $\mathbb{K}$-Nash groups reduces to the classification of 
quotients of locally $\mathbb{K}$-Nash structures over the additive group $\mathbb{K}^n$ by discrete subgroups. This also happens in the complex analytic context (see \cite[\S C.3.Cor.3]{Lojasewicz}). Here we prove that the relevant induced maps are  $\mathbb{K}$-Nash.

\begin{proposition}\label{quotients} 
\emph{(I)} Every connected $n$-dimensional abelian locally $\mathbb{K}$-Nash group is isomorphic to some
$(\mathbb{K}^n,+, f)/\Gamma$, where $\Gamma$ is 
a discrete subgroup of $\mathbb{K}^n$.

\emph{(II)} Let  $\alpha: (\mathbb{K}^n,+,\phi)$ and  $(\mathbb{K}^n,+,\psi)$  be  locally $\mathbb{K}$-Nash. Let $\Gamma _1$ and $\Gamma _2$  be discrete subgroups of $\mathbb{K}^n$.  Then the following are equivalent.
\begin{enumerate}
\item The locally $\mathbb{K}$-Nash groups $(\mathbb{K}^n,+,\phi )/\Gamma _1$ and $(\mathbb{K}^n,+,\psi )/\Gamma _2$ are isomorphic, and  
\item there exists an isomorphism $\alpha: (\mathbb{K}^n,+,\phi) \rightarrow (\mathbb{K}^n,+,\psi)$ such that $\alpha (\Gamma_1)=\Gamma_2$.
\end{enumerate}
\end{proposition}
\begin{proof}
(I) Let $G$ be a connected $n$-dimensional abelian locally $\mathbb{K}$-Nash group.
By Proposition\,\ref{covering}, $\widetilde{G}$ is a a simply connected $n$-dimensional abelian locally $\mathbb{K}$-Nash group, so we 
can apply Theorem\,\ref{T1}.

(II) We begin with (1) implies (2).
Let $\pi _1$ and $\pi _2$ denote the projections of $\mathbb{K}^n$ onto $\mathbb{K}^n/\Gamma _1$ and 
$\mathbb{K}^n/\Gamma _2$, respectively.
Let $\beta$ be the locally $\mathbb{K}$-Nash isomorphism from $(\mathbb{K}^n,+,\phi )/\Gamma _1$ to $(\mathbb{K}^n,+,\psi )/\Gamma _2$.
Take an open neighbourhood of the identity $U\sub \mathbb{K}^n$ and an analytic section 
$s_2 :\beta (\pi _1(U))\rightarrow \mathbb{K}^n$ such that $\pi _2  s_2= Id$ and $e\in s_2(\beta (\pi _1(U)))$.
Then, the map $d_e(s_2  \beta   \pi _1): T_e\,\mathbb{K}^n \rightarrow T_e\,\mathbb{K}^n$ is trivially a homomorphism of Lie 
algebras. 
Hence,  there exists a homomorphism of Lie groups, 
$\alpha :\mathbb{K}^n\rightarrow \mathbb{K}^n$,
such that $d_e \alpha =d_e (s_2  \beta   \pi _1)$. 
By symmetry, changing $\beta$ by $\beta ^{-1}$, we get that $\alpha$ is an isomorphism of Lie groups.
So $\alpha \in \GL_n(\mathbb{K})$ and $\beta   \pi _1=\pi _2  \alpha$.
Now, we note that both $\alpha$ and $\beta$ are injective maps and, hence, 
$\ker(\beta   \pi _1)=\Gamma _1$, $\ker(\pi _2  \alpha)=\alpha ^{-1}(\Gamma _2)$ and $\alpha (\Gamma _1)=\Gamma _2$.
It only remains to show that $\alpha$ is a locally $\mathbb{K}$-Nash map.
By Proposition\,\ref{covering}, the maps $\pi _1$ and $s_2$ are locally $\mathbb{K}$-Nash maps, so 
$\alpha |_{U}=s_2  \beta   \pi_1|_U$ is a locally $\mathbb{K}$-Nash map.
Thus, by Proposition\,\ref{lochomo}, $\alpha$ is a locally $\mathbb{K}$-Nash homomorphism.

For (2) implies (1), let  $\alpha: (\mathbb{K}^n,+,\phi) \rightarrow (\mathbb{K}^n,+,\psi)$ such that $\alpha (\Gamma_1)=\Gamma_2$, and 
$$
\beta :\mathbb{K}^n/\Gamma _1\rightarrow \mathbb{K}^n/\Gamma _2:u+\Gamma _1\mapsto \alpha (u)+\Gamma _2.
$$
By our assumption on $\alpha$, $\beta$ is an analytic isomorphism.
It only remains to show that $\beta$ is a locally $\mathbb{K}$-Nash map.
Take a sufficiently small open neighbourhood of the identity $U$ of $\mathbb{K}^n/\Gamma _1$ and an analytic section 
$s_1 :U\rightarrow \mathbb{K}^n$ such that $\pi _1  s_1= Id$.
We note that $\beta |_{U}=\pi _2  \alpha   s_1|_{U}$.
By Proposition\,\ref{covering}, the maps $s_1$ and $\pi _2$ are locally $\mathbb{K}$-Nash maps, so $\beta |_{U}$ is a locally 
$\mathbb{K}$-Nash map.
By Proposition\,\ref{lochomo}, $\beta$ is a locally $\mathbb{K}$-Nash map.
\end{proof}

We finish the paper by proving   one of our main results.

\begin{theorem}\label{algebraic category covering}
Every simply connected $n$-dimensional abelian locally $\mathbb{C}$-Nash group is the universal 
covering of some (abstract) abelian complex irreducible algebraic group of dimension $n$.
\end{theorem}
\begin{proof}
By Theorem\,\ref{T1}, we may assume that the locally $\mathbb{C}$-Nash group is of the form $(\mathbb{C}^n,+,\phi)$, for some meromorphic map 
$\phi =(\phi _1,\ldots ,\phi _n):\mathbb{C}^n \dashrightarrow \mathbb{C}^n$ admitting an AAT.
Moreover, by Fact\,\ref{T2}, we may assume that there exist a 
meromorphic function $\phi_0:\mathbb{C}^n \dashrightarrow \mathbb{C}$ and $R\in \mathbb{C}[X_0,\ldots, X_n]\setminus \{0\}$, irreducible,
such that $R(\phi_0,\ldots,\phi_n)= 0$ and, for each $\phi_i$, $i\in \{0,\dots,n\}$, there exist 
$R_i \in \mathbb{C}(X_{0}, \ldots, X_{2n+1})$ and $Q_i \in \mathbb{C}(X_{0},\ldots, X_{n})$ satisfying
$$
(\ast)\left\{ 
\begin{array}{l}
\phi_i(u+v)  = R_i (\phi_0(u),\ldots, \phi_n (u), \phi_0(v), \ldots,\phi_n(v)) \quad \text{and} \vspace{0.2cm}\\
\phi_i(-u)  = Q_i(\phi_0(u),\ldots,\phi_n(u)).
\end{array}
\right.
$$
Let $F$ be the field generated by the coefficients of $R,R_0,\ldots,R_n,Q_0,\ldots,Q_n$ and let $V\sub \mathbb{C}^{n+1}$ be  the zero set of $R$. Let us denote by $X$ the analytic subset of $\mathbb{C}^n$ such that $\mathbb{C}^n\setminus X$ is the set of points where all 
$\phi_0,\ldots, \phi_n$ are defined and let us consider the natural map
\[
\Phi:\mathbb{C}^n\setminus X \rightarrow V: u \mapsto (\phi_0(u),\ldots,\phi_n(u)) .
\]
Note that, since $\phi_1,\ldots,\phi_n$ are algebraically independent over $\mathbb{C}$, the affine irreducible algebraic 
subset $V$ of $\mathbb{C}^{n+1}$ is $n$-dimensional.

\smallskip
\emph{Step $1$}. Our  first aim is to show that $V$ is a pre-group in the sense of Weil  \cite[\S\,I.1]{Weil}, via the group $(\mathbb{C}^n,+,\phi)$. That is, we will prove that there is a rational map $f:V\times V \dashrightarrow V$ such that:

\smallskip
\noindent{(G1)} If $x,y$ are independent generic points of $V$ over  the field $F$ and $z=f(x,y)$ then 
\[
F(x)\sub F(z,y) \quad \text{ and  } \quad F(y)\sub F(x,z).
\]
\smallskip
{(G2)} If $x,y,t$ are independent generic points of $V$ over the field $F$ then 
\[
f(f(x,y),t)=f(x,f(y,t)). 
\]

Indeed, we consider 
$$
f:V\times V  \dashrightarrow  \C^{n+1}: (x,y)  \mapsto  (R_0(x,y),\ldots,R_n(x,y))$$
and
$$g:V  \dashrightarrow  \C^{n+1}:  x  \mapsto  (Q_0(x),\ldots,Q_n(x)),$$
which are rational maps over $F$. As in the meromorphic context, we use the symbol $\dashrightarrow$ to stress that these functions are not global, they are defined off an algebraic subset. We first note that the image of both $f$ and $g$ is contained in $V$. For, take $D$ the Zariski open subset of $V\times V$ of all points $(x,y)$ where $f$ is defined. The preimage $f^{-1}(V)$ is a Zariski closed subset of $D$, so there is an algebraic subset $Y$ of $V\times V$ such that $f^{-1}(V)=D\cap Y$. Since $D\cap Y$ contains $\Phi(\mathbb{C}\setminus X)\times \Phi(\mathbb{C}\setminus X)$ by $(\ast)$, we deduce that $Y$ has dimension $2n$ and, therefore, it equals the irreducible $V\times V$, so that $f^{-1}(V)=D\cap Y=D$, as required. For $g$ the argument is similar.

Let us show that $f$ satisfies property {G1}. Indeed, take $D_0$ the Zariski open subset of $V\times V$ of all points $(x,y)$ where $f(f(x,y),g(x))$ is defined and consider the subset 
of $D_0$ given by
\[
\{(x,y)\in D_0 \suchthat f(f(x,y),g(x))=y \}.
\]
The set is Zariski closed in $D_0$ and arguing as above we get that it equals $D_0$. This shows that $F(y)\sub F(x,z)$ and, by symmetry, we get that $F(x)\sub F(z,y)$, as required.
The proof of {G2} is similar.

\smallskip
\emph{Step} 2.
By \cite[Prop.\,4]{Weil}, there exists a birational equivalence $\omega:V\dashrightarrow W$, where $W$ is an affine variety 
group-chunk (see the definition just before \cite[Prop.\,3]{Weil}).
By \cite[Thm., p.$375$]{Weil} and the beginning of the proof in \cite[\S\,III.6]{Weil}, there exists an algebraic 
group $G$ with a Zariski chart of the form $\rho:W_1\rightarrow W$, such that:
\begin{enumerate}
 \item We have the following diagram:
\begin{displaymath}
    \xymatrix{ \mathbb{C}^n\setminus X \ar[d]^\Phi & \hspace{1cm}W_1\sub G \ar[d]^\rho \\
               V \ar[r]^\omega & W \ar@/^/@{.>}[u]^{\rho^{-1}} }.
\end{displaymath}
\item The following equalities hold:
\begin{align*}
f(x,y) & = (\rho^{-1}  \omega)^{-1}\big( (\rho^{-1}  \omega)(x)\cdot(\rho^{-1}  \omega)(y) \big)\\
g(x) & =(\rho^{-1}  \omega)^{-1}\big([(\rho^{-1}  \omega)(x)]^{-1}\big),
\end{align*}
for all points $x,y\in V$ on which the involved functions can be evaluated.
\end{enumerate}
We also note that $G$ is abelian because the algebraic subset of the irreducible $G\times G$ which consists in commuting elements does 
contain a subset of dimension $2n$. 

\emph{Step} 3.
Now, our purpose is to define a local isomorphism between $\mathbb{C}^n$ and $G$. 
Let $Z$ be the algebraic subset of $V$ for which the birational maps $\omega$ and $g$ are well-defined maps in $V\setminus Z$. 
Consider also the analytic subset $\Phi^{-1}(Z)$ of $\mathbb{C}^n\setminus X$ and denote 
$U:=(\mathbb{C}^n\setminus X)\setminus \Phi^{-1}(Z)$, which is an open dense subset of $\mathbb{C}^n\setminus X$. 
We also note that $U$ is connected, since it is the complement in $\mathbb{C}^n$ of a (thin) analytic subset. 
Therefore, we can consider the following analytic map:
\[
h:=(\rho^{-1}  \omega)   \Phi:U \rightarrow W_1\sub G.
\]

\emph{Step} 3.1.
We claim that the map $h$ satisfies that $h(x+y)=h(x)h(y)$, for all $x,y\in U$ with $x+y\in U$. Again, the set $D_1:=\{(x,y)\in U\times U: x+y\in U\}$ is an open subset of the connected $U\times U$ whose complement is  the preimage of an analytic subset, hence analytic. Thus, $D_1$ is connected. On the other hand, 
the analytic subset $\{(x,y)\in D_1:h(x+y)=h(x)h(y)\}$ has non-empty interior in $D_1$ because of the above relation (2) of $f$ and 
the group operation of $G$.
Since $D_1$ is connected, they must coincide, as required.
Similarly, we have that $h(-x)=h(x)^{-1}$, for all $x\in U$ with $-x\in U$.

\emph{Step} 3.2. As $U$ is dense in $\mathbb{C}^n$ and the maps $\mathbb{C}^n\rightarrow \mathbb{C}^n:x\mapsto 2x$ and 
$\mathbb{C}^n\rightarrow \mathbb{C}^n:x\mapsto -x$ are homeomorphisms, 
the sets $\{x\in \mathbb{C}^n \suchthat 2x\in U\}$ and $\{x\in \mathbb{C}^n \suchthat -x\in U\}$ are dense subsets of $\mathbb{C}^n$. 
The subset of points of $U$ where $\phi$ is a local analytic diffeomorphism is a proper analytic subset of $U$. 
Thus, there exists $a\in U$ such that $2a\in U$, $-a\in U$ and $\phi$ is a local diffeomorphism at $a$. 
It follows that there exists an open semialgebraic neighbourhood $U_1$ of $a$ for which we have that $U_1+U_1\sub U$ and such that 
$\phi|_{U_1}$ is an analytic diffeomorphism onto its image. Consider the following analytic diffeomorphism onto its image,
\[
\begin{array}{rcl}
s: -a+U_1 & \rightarrow & G\\
-a+u  &  \mapsto &  h(a)^{-1}h(u).
  \end{array}
\]

\emph{Step} 3.3. We claim that $s$ is a local isomorphism. 
Indeed, given points $u,v\in U_1$ for which $(-a+u)+(-a+v)\in -a+U_1$, we have that 
$-a+u+v\in U_1$ and by definition
\begin{align*}
s\big((-a+u)+(-a+v) \big) & =s\big(-a+(-a+u+v) \big)\\
& =h(a)^{-1}h(-a+u+v).
\end{align*}
As $-a\in U$ and $u+v\in U_1+U_1\sub U$, it follows from Step 3.1 that
\begin{align*}
s\big((-a+u)+(-a+v) \big) &=h(a)^{-1}\big(h(-a)h(u+v)\big)\\
& =h(a)^{-1}\big(h(-a)h(u)h(v)\big)\\
& =h(a)^{-1}\big(h(a)^{-1}h(u)h(v)\big)\\
& =h(a)^{-1}h(u)h(a)^{-1}h(v)\\
& =s(-a+u)s(-a+v),
\end{align*}
as required. 

\emph{Step} 3.4.
By Corollary\,\ref{C-Nash structure}, we have that $G$ has a locally $\mathbb{C}$-Nash structure and our aim now is to show that $s$ 
is a locally $\mathbb{C}$-Nash map. 
We explicitly define charts of $G$ and $(\mathbb{C}^n,+,\phi)$ to work with. 
Note that $h(a)^{-1}W_1\rightarrow W:y \mapsto \rho(h(a)y)$ is also a Zariski chart. 
In fact, since $h(a)\in W_1$, it is a Zariski chart of the identity. 
By Corollary\,\ref{C-Nash structure}, there exist an open subset $W_2$ of $W_1$, with $h(a)\in W_2$, and a projection $\pi$ such that 
\[
h(a)^{-1}W_2\rightarrow \pi(\rho(W_2)):y\mapsto \pi(\rho(h(a)y))
\]
is a chart of the identity of the locally $\mathbb{C}$-Nash structure of $G$, as required. 
We also consider the following chart of $0$ of $(\mathbb{C}^n,+,\phi)$,
\[
-a+U_1\rightarrow \phi(U_1): y \mapsto \phi(a+y),
\]
where we have shrunk $U_1$ so that $s(-a+U_1)\sub h(a)^{-1}W_2$.

To show that the map $s$ is a locally $\mathbb{C}$-Nash map, using the above charts, it reduces to show that the map
\[
\phi(U_1)\rightarrow \pi\big(\rho(W_2)\big): z \mapsto \pi(\rho(h(\phi^{-1}(z)))) 
\]
is algebraic over $\mathbb{C}(\id)$. 
By definition of the map $h$, we have that
\[
\pi(\rho(h(\phi^{-1}(z))))=\pi(\omega(\phi_0(\phi^{-1}(z)),z)). 
\]
On the other hand, $\phi_0$ is algebraic over $\mathbb{C}(\phi)$ and, hence, $\phi_0  \phi^{-1}$ is algebraic over $\mathbb{C}(\id)$. 
Thus, since $\omega$ is a rational map and $\pi$ is a projection, we deduce that $s$ is a locally $\mathbb{C}$-Nash map, as required.

\emph{Step} 3.5.
Finally, 
 by the monodromy theorem the local isomorphism $s$ can be lifted to an analytic global isomorphism 
$S:(\mathbb{C}^n,+,\phi)\rightarrow \widetilde{G}$. 
 Thus showing   that $S$ is a locally $\mathbb{C}$-Nash map reduces to check that the map $s$ is a locally $\mathbb{C}$-Nash 
(see Proposition\,\ref{lochomo}), so we are done.
\end{proof} 

From the latter theorem, Lemma\,\ref{C-Nash structure}  and Proposition\,\ref{covering} we get the following.

\begin{cor}\label{equalcat} The category of simply connected  abelian locally $\mathbb{C}$-Nash groups coincides with  that of universal coverings of the abelian complex irreducible algebraic groups.
\end{cor}

\begin{rem} In the proof of Theorem\,\ref{algebraic category covering}, if $f$ is invariant under complex conjugation ($\overline{f(\overline{x})}=f(x)$ for all $x\in V\times V$) then the algebraic group we obtain is defined over $\R$. Therefore, we also have an alternative proof -- in the abelian case --  to the result in \cite{Hrushovski_Pillay}  mentioned in the Introduction.
\end{rem}

\bibliographystyle{plain}
\bibliography{biblioLCN}

\begin{thebibliography}{10}

\bibitem{Adamus_Randriambololona}
J.~Adamus and S.~Randriambololona.
\newblock Tameness of holomorphic closure dimension in a semialgebraic set.
\newblock {\em Math. Ann.}, 355(3):985--1005, 2013.

\bibitem{BDOAAT}
E.~Baro, J.~de~Vicente, and M.~Otero.
\newblock An extension result for maps admitting an algebraic addition theorem.
\newblock {\em To appear in Journal of Geometric Analysis. DOI
  10.1007/s12220-018-9992-7.
  \href{https://arxiv.org/pdf/1704.08514.pdf}{ArXiv:1704.08514}}, 2017.

\bibitem{BDO2ALKNG}
E.~Baro, J.~de~Vicente, and M.~Otero.
\newblock Two-dimensional abelian locally $\mathbb{K}$-{Nash} groups.
\newblock Preprint, 2017.

\bibitem{Bochnak_Coste_Roy}
J.~Bochnak, M.~Coste, and M.F. Roy.
\newblock {\em Real algebraic geometry}, volume~36 of {\em Ergebnisse der
  Mathematik und ihrer Grenzgebiete (3) [Results in Mathematics and Related
  Areas (3)]}.
\newblock Springer-Verlag, Berlin, 1998.

\bibitem{tesis}
J.~de~Vicente.
\newblock Locally {Nash} groups.
\newblock {\em Ph.D. thesis, Universidad A\'utonoma de Madrid,
  \href{http://biblioteca.uam.es}{biblioteca.uam.es}}, 2017.

\bibitem{Fernando_Gamboa_Ruiz}
J.F. Fernando, J.M. Gamboa, and J.M. Ruiz.
\newblock Finiteness problems on {N}ash manifolds and {N}ash sets.
\newblock {\em J. Eur. Math. Soc. (JEMS)}, 16(3):537--570, 2014.

\bibitem{FlodorG04}
P.~Flondor and F.~Guaraldo.
\newblock On {N}ash groups.
\newblock {\em Math. Rep. (Bucur.)}, 6(56)(3):225--232, 2004.

\bibitem{FLR87}
E.~Fortuna, S.~\L~\!\!ojasiewicz, and M.~Raimondo.
\newblock Alg\'ebricit\'e de germes analytiques.
\newblock {\em J. Reine Angew. Math.}, 374:208--213, 1987.

\bibitem{Gunning_Rossi}
R.C. Gunning and H.~Rossi.
\newblock {\em Analytic functions of several complex variables}.
\newblock Prentice-Hall, Inc., Englewood Cliffs, N.J., 1965.

\bibitem{Hrushovski}
E.~Hrushovski.
\newblock {\em Contributions to stable model theory}.
\newblock PhD thesis, Berkeley, 1986.

\bibitem{Hrushovski_Pillay}
E.~Hrushovski and A.~Pillay.
\newblock Groups definable in local fields and pseudo-finite fields.
\newblock {\em Israel J. Math.}, 85(1-3):203--262, 1994.

\bibitem{Hrushovski_Pillay_Errata}
E.~Hrushovski and A.~Pillay.
\newblock Affine {N}ash groups over real closed fields.
\newblock {\em Confluentes Math.}, 3(4):577--585, 2011.

\bibitem{Huber_Knebusch}
R.~Huber and M.~Knebusch.
\newblock A glimpse at isoalgebraic spaces.
\newblock {\em Note Mat.}, 10(suppl. 2):315--336, 1990.

\bibitem{Kawakami96}
Tomohiro Kawakami.
\newblock Nash {$G$} manifold structures of compact or compactifiable
  {$C^\infty G$} manifolds.
\newblock {\em J. Math. Soc. Japan}, 48(2):321--331, 1996.

\bibitem{Knebusch}
M.~Knebusch.
\newblock Isoalgebraic geometry: first steps.
\newblock In {\em Seminar on {N}umber {T}heory, {P}aris 1980-81 ({P}aris,
  1980/1981)}, volume~22 of {\em Progr. Math.}, pages 127--141. Birkh\"auser,
  Boston, Mass., 1982.

\bibitem{Lojasewicz}
Stanis\l~aw \L~\!\!ojasiewicz.
\newblock {\em Introduction to complex analytic geometry}.
\newblock Birkh\"auser Verlag, Basel, 1991.
\newblock Translated from the Polish by Maciej Klimek.

\bibitem{Madden_Stanton}
J.J. Madden and C.M. Stanton.
\newblock One-dimensional {N}ash groups.
\newblock {\em Pacific J. Math.}, 154(2):331--344, 1992.

\bibitem{Mumford}
David Mumford.
\newblock {\em Algebraic geometry. {I}}.
\newblock Springer-Verlag, Berlin-New York, 1976.
\newblock Complex projective varieties, Grundlehren der Mathematischen
  Wissenschaften, No. 221.

\bibitem{Perrin}
D.~Perrin.
\newblock Groupes hens\'eliens.
\newblock {\em Bull. Soc. Math. France}, 104(4):369--381, 1976.

\bibitem{Peterzil_Starchenko_1}
Y.~Peterzil and S.~Starchenko.
\newblock Expansions of algebraically closed fields in o-minimal structures.
\newblock {\em Selecta Math. (N.S.)}, 7(3):409--445, 2001.

\bibitem{Peterzil_Starchenko_2}
Y.~Peterzil and S.~Starchenko.
\newblock Expansions of algebraically closed fields. {II}. {F}unctions of
  several variables.
\newblock {\em J. Math. Log.}, 3(1):1--35, 2003.

\bibitem{Poizat}
B.~Poizat.
\newblock An introduction to algebraically closed fields \& varieties.
\newblock In {\em The model theory of groups ({N}otre {D}ame, {IN},
  1985--1987)}, volume~11 of {\em Notre Dame Math. Lectures}, pages 41--67.
  Univ. Notre Dame Press, Notre Dame, IN, 1989.

\bibitem{Rosenlicht}
Maxwell Rosenlicht.
\newblock Some basic theorems on algebraic groups.
\newblock {\em Amer. J. Math.}, 78:401--443, 1956.

\bibitem{Shiota1}
M.~Shiota.
\newblock {\em Nash manifolds}, volume 1269 of {\em Lecture Notes in
  Mathematics}.
\newblock Springer-Verlag, Berlin, 1987.

\bibitem{Shiota2}
M.~Shiota.
\newblock Nash functions and manifolds.
\newblock In {\em Lectures in real geometry ({M}adrid, 1994)}, volume~23 of
  {\em De Gruyter Exp. Math.}, pages 69--112. de Gruyter, Berlin, 1996.

\bibitem{Silhol}
Robert Silhol.
\newblock {\em Real algebraic surfaces}, volume 1392 of {\em Lecture Notes in
  Mathematics}.
\newblock Springer-Verlag, Berlin, 1989.

\bibitem{Tancredi_Tognoli}
A.~Tancredi and A.~Tognoli.
\newblock On the relative {N}ash approximation of analytic maps.
\newblock {\em Rev. Mat. Complut.}, 11(1):185--200, 1998.

\bibitem{Varadarajan}
V.~S. Varadarajan.
\newblock {\em Lie groups, {L}ie algebras, and their representations}, volume
  102 of {\em Graduate Texts in Mathematics}.
\newblock Springer-Verlag, New York, 1984.

\bibitem{Weil}
A.~Weil.
\newblock On algebraic groups of transformations.
\newblock {\em Amer. J. Math.}, 77:355--391, 1955.

\bibitem{Wilkie05}
A.~J. Wilkie.
\newblock Covering definable open sets by open cells.
\newblock {\em O-minimal Structures, Proceedings of the RAAG Summer School
  Lisbon 2003, Lecture Notes in Real Algebraic and Analytic Geometry.
  Cuvillier}, 2005.

\end{thebibliography}
\end{document}